\documentclass[preprint,12pt]{elsarticle}

\usepackage{amssymb}
\usepackage{amsthm}
\usepackage{amsmath}
\usepackage[pagewise]{lineno}

\newtheorem{definition}{Definition}[section]
\newtheorem{theorem}{Theorem}[section]

\newtheorem{remark}{Remark}[section]
\newtheorem{lemma}{Lemma}[section]
\newtheorem{assumption}{Assumption}[section]

\journal{Journal Differential Equations}

\begin{document}

\begin{frontmatter}



\title{Mean-square invariant manifolds for ill-posed stochastic evolution equations driven by nonlinear noise}


\author[label1]{Zonghao Li}
\author[label1]{Caibin Zeng\corref{cor1}}
\ead{macbzeng@scut.edu.edu.cn}
\author[label2]{Jianhua Huang}
\address[label1]{School of Mathematics, South China University of Technology, Guangzhou 510640, China}
\address[label2]{College of Science, National University of Defense Technology, Changsha 410073, China}
\cortext[cor1]{Corresponding author.}

\begin{abstract}
 This paper discerns the invariant manifold of a class of ill-posed stochastic evolution equations driven by a nonlinear multiplicative noise. To be more precise, we establish the existence of mean-square random unstable invariant manifold and only mean-square stable invariant set. Due to the lack of the Hille-Yosida condition, we construct a modified variation of constants formula by the resolvent operator. With the price of imposing an unusual condition involving a non-decreasing map, we set up the Lyapunov-Perron method and derive the required estimates. We also emphasize that the Lyapunov-Perron map in the foward time loses the invariant due to the adaptedness, we alternatively establish the existence of mean-square random stable sets.
\end{abstract}



\end{frontmatter}


%\tableofcontents

\section{Introduction}
In this paper, we shall study the existence of mean-square random stable and unstable invariant manifolds of the following It\^{o} stochastic evolution equation defined in a separable Hilbert space $X$,
\begin{equation}\label{eq1.1}
	\left\{
	\begin{array}{l}
		du=Audt+F\left(u\right)dt+\sigma(u){dW},\\
		u\left(0\right)=u_0\in\overline{D\left(A\right)}=X_0. 	\\
	\end{array}\right.
\end{equation}  
where $A:D\left(A\right)\subset{X}\to{X}$ is a linear operator whose domain is non-densely defined. Indeed, $W$ is a two-sided cylindrical Wiener process on a complete filtered probability space $\left({\Omega}, \mathcal{F}, \{{\mathcal{F}_t}\}_{t\in\mathbb{R}}, \mathbb{P}\right)$. While $F$ and $\sigma$ are nonlinear functions with some natural imposed conditions (specified later). Note that the Cauchy problem to \eqref{eq1.1} is ill-posed since the Hille-Yosida theory for $C_0$-semigroup generated by $A$ breaks down. Cauchy problems with a non-dense domain cover several types of differential equations, including delay differential equations, age-structured models and some  evolution equations with nonlinear boundary conditions. To overcome the embarrassment, the integrated semigroup theory allows us to define a suitable mild solution, or so-called integrated solution. 
The concept of integrated semigroups was first introduced by Arendt \cite{WA1987,WA2001} and it was applied to study the existence and uniqueness of solutions to such non-homogeneous Cauchy problems in a deterministic setting by Da Prato and Sinestrari \cite{DS1987}. Later on, Thieme \cite{HS1990_,HS1990} established the well-posedness of non-autonomous and semilinear Cauchy problems under a prior Hille-Yosida type estimate for the resolvent of non-densely defined operators. By discarding the mentioned estimate, Magal and Ruan \cite{PM2007,PM_2009,PM2009} improved the above results and applied the obtained center manifold theorem to study Hopf bifurcation of age-structured models. Recently, Neam\c{t}u \cite{AN2020} extended the mentioned theory of integrated semigroups to ill-posed stochastic evolution equations \eqref{eq1.1} driven by linear white noise ($\sigma\left(u\right)=u$) and established the existence of random stable/unstable manifolds based on the Lyapunov-Perron method. Following this direction, Zeng and Shen \cite{ZS2021} obtained the existence of the invariant foliations of \eqref{eq1.1}. Moreover, Li and Zeng \cite{LZ2021} studied the center manifold issue of \eqref{eq1.1} with exponential trichotomy. Therefore, the mentioned work brings much interest to discern the long-time dynamics of \eqref{eq1.1}. 

On the other hand, the above-mentioned results are mainly about the pathwise random invariant manifolds of the equation. A pathwise random invariant manifolds is a set-valued random variable that is given by a graph of a Lipschitz map. To study such manifolds, we need to define pathwise random dynamical systems for the stochastic evolution equation. There are several results regarding invariant manifolds in the framework of random dynamical systems. Mohammed and Scheutzow \cite{Mh1999} studied the existence of local stable and unstable manifolds of stochastic differential equations driven by semimartingales. Duan \emph{et al.} \cite{Duan2004}, Lu and Schmalfuss \cite{KB2007} and Caraballo \emph{et al.} \cite{Cara2010} studied stable and unstable manifolds for stochastic partial differential equations. Lian and Lu \cite{LL2010} proved a multiplicative ergodic theorem and then use this theorem to establish the stable and unstable manifold theorem for nonuniformly hyperbolic random invariant sets. Li \emph{et al.} \cite{PJ2013} proved the persistence of smooth normally hyperbolic invariant manifolds for dynamical systems under random perturbations. As for center manifolds, Chen \emph{et al.} \cite{Chen2015} studied center manifolds for stochastic partial differential equations under an assumption of exponential trichotomy. Shi \cite{Shi2020} studied the limiting behavior of center manifolds for a class of singularly perturbed stochastic partial differential equations in terms of the phase spaces.

To our knowledge, the existence of pathwise random dynamical systems for stochastic partial differential equations is mainly restricted to additive white noise and linear multiplicative white noise. In this case, the key step is to transform the stochastic partial differential equations into random ones. For instance, a coordinate transform based on the stationary Ornstein-Uhlenbeck process was used to transform \eqref{eq1.1} driven by linear noise to a deterministic equation with random coefficients \cite{Cara2013,Duan2003,Duan2004}. Also, there were a few shots for the case of nonlinear multiplicative noise. To this aspect, Garrido-Atienza \emph{et al.} \cite{GALS2010} studied the local unstable manifold for stochastic PDEs with a fractional Brownian motion with $1/2<H<1$ using fractional calculus technique and constructing a stopping time sequence. By means of rough paths theory, Neam\c{t}u and Kuehn \cite{KN2021} established the existence and regularity of local center manifolds for rough differential equations. 

Our aim here is to develop a theory of mean-square random invariant manifolds of \eqref{eq1.1}  based on the so-called mean-square random dynamical system, presented by Lorenz and Kloeden \cite{LK2012}. Indeed, we adopt Wang's approach in \cite{BA2021} and integrated semigroup theory. Precisely, we will define a backward and forward stochastic equation, respectively,  and both of them involve the conditional expectation $\mathbb{E}\left(\cdot|\mathcal{F}_t\right)$ since every terms of them is required to be $\mathcal{F}_t$-adapted. Then we use the Lyapunov-Perron method to construct mean-square random unstable and stable invariant manifolds, respectively. By solving these two equations in the space of stochastic processes defined on backward time and forward time respectively and defining random sets of this solutions, we can prove the existence of the invariant manifold which is given by a graph of a Lipchitz map. But it turns out that the invariance of random set defined by solutions of forward one can not be proved. Alternatively, we establish the existence of mean-square random stable sets which is given by a graph of Lipschitz map which maps from a subspace instead of the whole space.

It also needs to emphasize that the standard variation of constants formula is not applicable due to the non-densely defined operator $A$. Even worse, Young's convolution inequality is not available such that the Gronwall-type lemma fails to deduce estimates of the solutions. For the former, we will construct a new variation of constants formula by the resolvent operator  connecting $\overline{D\left(A\right)}$ and $X$. While for the latter, we will impose an additional condition to complete the required estimates.

The rest of this paper is organized as follows. In Section \ref{sec2}, we collect some basic definitions of mean-square random dynamical systems and integrated semigroups and lay out the basic assumptions. Then in Section \ref{sec3} and Section \ref{sec4}, we prove the existence of mean-square random unstable invariant manifolds and stable invariant sets for \eqref{eq1.1}, respectively. Finally in Section \ref{sec5}, we summarize the conclusions obtained and discuss several possible extensions.  
\section{Preliminaries}\label{sec2}
\subsection{Mean-square random dynamical systems}
Suppose $V$ is a Banach space with norm ${\left\Vert\cdot\right\Vert}_V$. Let $\left({\Omega}, \mathcal{F}, \{{\mathcal{F}_t}\}_{t\in\mathbb{R}}, \mathbb{P}\right)$ be a complete filtered probability space satisfying the usual condition, that is, $\left\{{\mathcal{F}_t}\right\}_{t\in\mathbb{R}}$ is an increasing right continuous family of sub-$\sigma$-algebras of $\mathcal{F}$ that contains all $\mathbb{P}$-null sets. Given $p\in\left(1,+\infty\right)$ and $t\in\mathbb{R}$, denote by $L^p\left(\Omega, \mathcal{F}_t; V\right)$ the subspace of $L^p\left(\Omega, \mathcal{F}; V\right)$, which consists of all strongly $\mathcal{F}_t$-measurable functions $\psi$ in $L^p\left(\Omega, \mathcal{F}; V\right)$. We now present the definition of mean dynamical system over $\left({\Omega}, \mathcal{F}, \{{\mathcal{F}_t}\}_{t\in\mathbb{R}}, \mathbb{P}\right)$(see \cite{LK2012,BA2021}).
\begin{definition}
	A family $\Phi=\left\{\Phi\left(t,\tau\right):t\in\mathbb{R}^+, \tau\in\mathbb{R}\right\}$ of mappings is called a mean-square random dynamical system on $L^2\left(\Omega, \mathcal{F}; V\right)$ over the space $\left({\Omega}, \mathcal{F}, \{{\mathcal{F}_t}\}_{t\in\mathbb{R}}, \mathbb{P}\right)$ if $\forall$ $\tau\in\mathbb{R}$, $t,s\in\mathbb{R}^+$, it follows that
	\begin{itemize}
		\item [$\left(\mathrm{i}\right)$]$\Phi\left(t,\tau\right)$ maps $L^2\left(\Omega, \mathcal{F}_\tau; V\right)$ to $L^2\left(\Omega, \mathcal{F}_{t+\tau}; V\right)$;
		\item [$\left(\mathrm{ii}\right)$]$\Phi(0,\tau)$ is the identity operator on $L^2\left(\Omega, \mathcal{F}_\tau; V\right)$;
		\item [$\left(\mathrm{iii}\right)$]$\Phi(t+s,\tau)=\Phi(t,\tau+s)\circ\Phi(s,\tau)$.   
	\end{itemize}	  
\end{definition}     
Then we present the definition of the invariance property.
\begin{definition}\label{inv}
	Suppose $\Phi=\left\{\Phi\left(t,\tau\right):t\in\mathbb{R}^+, \tau\in\mathbb{R}\right\}$ is a mean-square random dynamical system on $L^2\left(\Omega, \mathcal{F}; V\right)$ over $\left({\Omega}, \mathcal{F}, \{{\mathcal{F}_t}\}_{t\in\mathbb{R}}, \mathbb{P}\right)$ and $\mathcal{M}=\{\mathcal{M}\left(\tau\right)\subseteq{L^2\left(\Omega,\mathcal{F}_\tau;V\right)}:\tau\in\mathbb{R}\}$ is a family of subsets of $L^2\left(\Omega,\mathcal{F}; V\right)$. Then $\mathcal{M}$ is called a mean-square random invariant set of $\Phi$ in $V$ if for all $t\in{\mathbb{R}}^+$,
	\begin{equation*}
		\Phi\left(t,\tau\right)\mathcal{M}\left(\tau\right)\subseteq{\mathcal{M}\left(\tau+t\right)}.
	\end{equation*}	   
\end{definition}

We list a part of basic assumptions here. Let $Y$ be another separable Hilbert space. The norm of $X$ is denoted by $\left\Vert\cdot\right\Vert$. Let $\mathcal{L}\left(Y,X\right)$ be the space of continuous linear operators from $Y$ to $X$. Given a symmetric and nonnegative operator $Q\in{\mathcal{L}\left(Y,X\right)}$, we write $Y_0=Q^{1/2}Y$. The space of Hilbert-Schmidt operators from $Y_0$ to $X$ is denoted by $\mathcal{L}_2\left(Y_0,X\right)$ with norm ${\left\Vert\cdot\right\Vert}_{\mathcal{L}_2\left(Y_0,X\right)}$. From now on we assume that $W$ is a two-sided $Y$-valued $Q$-Wiener process defined on $\left({\Omega}, \mathcal{F}, \{{\mathcal{F}_t}\}_{t\in\mathbb{R}}, \mathbb{P}\right)$. For simplicity, we denote $L^2\left(\Omega,\mathcal{F};X\right)$ and its norm as $L^2\left(\Omega,X\right)$ and ${\left\Vert\cdot\right\Vert}_{L^2\left(\Omega,X\right)}$, respectively.

\begin{assumption}\label{as1}
	Assume that $F:L^2\left(\Omega,X_0\right)\to{L^2\left(\Omega,X\right)}$ is globally Lipschitz continuous, i.e.
	\begin{equation}\label{eq1}
		{\left\Vert{F\left(u_1\right)-F\left(u_2\right)}\right\Vert}_{L^2\left(\Omega,X\right)}\leq{L_1}{\left\Vert{u_1-u_2}\right\Vert}_{L^2\left(\Omega,X_0\right)}.
	\end{equation}
	for all $u_1,u_2\in{L^2\left(\Omega,X_0\right)}$, where $L_1$ is the Lipschitz constant and $F\left(0\right)=0$. Thus
	\begin{equation}\label{eq_1}
		{\left\Vert{F\left(u\right)}\right\Vert}_{L^2\left(\Omega,X\right)}\leq{L_1}{\left\Vert{u}\right\Vert}_{L^2\left(\Omega,X_0\right)}. 
	\end{equation}
	for all $u\in{L^2\left(\Omega,X_0\right)}$.
	For the nonlinear diffusion term, we assume that $\sigma:L^2\left(\Omega,X_0\right)\to{L^2\left(\Omega,\mathcal{L}_2\left(Y_0,X\right)\right)}$ is globally Lipschitz continuous, i.e. 
	\begin{equation}\label{eq2}
		{\left\Vert{\sigma\left(u_1\right)-\sigma\left(u_2\right)}\right\Vert}_{L^2\left(\Omega,\mathcal{L}_2\left(Y_0,X\right)\right)}\leq{L_2}{\left\Vert{u_1-u_2}\right\Vert}_{L^2\left(\Omega,X_0\right)}.
	\end{equation}
	for all $u_1,u_2\in{L^2\left(\Omega,X_0\right)}$, where $L_2$ is the Lipschitz constant and $\sigma\left(0\right)=0$. Thus
	\begin{equation}\label{eq_2}
		{\left\Vert{\sigma\left(u\right)}\right\Vert}_{L^2\left(\Omega,\mathcal{L}_2\left(Y_0,X\right)\right)}\leq{L_2}{\left\Vert{u}\right\Vert}_{L^2\left(\Omega,X_0\right)}. 
	\end{equation}
	for all $u\in{L^2\left(\Omega,X_0\right)}$.
\end{assumption}
\subsection{Integrated semigroups}
In this subsection, we collect some basic results about non-densely defined operator and integrated semigroup from \cite{PM2007,PM2009,PM_2009,Magal2016}. The resolvent set of $A$ is denoted by $\rho\left(A\right)=\left\{\lambda\in\mathbb{C}:\lambda{I}-A \text{~is invertible}\right\}$ and we write $R_\lambda\left(A\right):=\left(\lambda{I}-A\right)^{-1}$. The spectrum of $A$ is denoted by $\sigma\left(A\right):=\mathbb{C}\backslash\rho\left(A\right)$. And we construct the part of $A$ denoted by $A_0:D\left(A_0\right)\subset{X_0}\to{X_0}$, which is a linear operator on $X_0$ defined by 
\begin{equation*}
	A_0x=Ax,\forall{x\in{D\left(A_0\right)}:=\left\{y\in{D\left(A\right):Ay\in{X_0}}\right\}}.	
\end{equation*}
Assume that there exists a constant $\vartheta$ satisfying $\left(\vartheta,+\infty\right)\subset\rho\left(A\right)$, which means $\rho\left(A\right)\neq\emptyset$. Then it follows from \cite[Lemma 2.1]{PM_2009} that $\rho\left(A\right)=\rho\left(A_0\right)$, from which we get $\sigma\left(A\right)=\sigma\left(A_0\right)$ and for each $\lambda>\vartheta$, 
\begin{equation*}
	D\left(A_0\right)=R_\lambda\left(A\right)X_0,  ~{R_\lambda\left(A_0\right)=R_\lambda\left(A\right)|}_{X_0}. 
\end{equation*} 

Recall that $A$ is said to be a Hille-Yosida operator if there exist two constants, $\vartheta\in\mathbb{R}$ and $M\geq1$, such that $\left(\vartheta,+\infty\right)\subset\rho\left(A\right)$ and
\begin{equation*}
	\left\Vert{\left(\lambda{I}-A\right)}^{-k}\right\Vert_{\mathcal{L}\left(X\right)}\leq\frac{M}{{\left(\lambda-\vartheta\right)}^{k}},~\forall\lambda>\vartheta,~\forall{k}\geq1.
\end{equation*} 
In the following, we assume $A$ satisfies some weaker conditions:
\begin{assumption}\label{as2.1}
	Let $\left(X,\left\Vert\cdot\right\Vert\right)$ be a Banach space and let $A:D\left(A\right)\subset{X}\to{X}$ be a linear operator. Assume that 
	\begin{itemize}
		\item [$\left(a\right)$] $A$ is a Hille-Yosida operator on $X_0$;
		\item[$\left(b\right)$]$\lim_{\lambda\to+\infty}R_\lambda\left(A\right)x=0,~\forall{x}\in{X}.$
	\end{itemize}
\end{assumption}
\begin{lemma}\cite[Lemma 2.1]{PM2007}\label{lem2.3}
	Let $\left(X,\left\Vert\cdot\right\Vert\right)$ be a Banach space and $A:D\left(A\right)\subset{X}\to{X}$ be a linear operator. Assume that there exists $\vartheta\in{\mathbb{R}}$, such that $\left(\vartheta,+\infty\right)\subset\rho\left(A\right)$ and 
	\begin{equation*}
		\limsup_{\lambda\to+\infty}\lambda{\left\Vert{R_\lambda\left(A\right)}\right\Vert}_{\mathcal{L}\left(X_0\right)}<+\infty.
	\end{equation*}
	Then the following assertions are equivalent:
	\begin{itemize}
		\item [$\left(\mathrm{i}\right)$] $\lim_{\lambda\to+\infty}\lambda{R_\lambda\left(A\right)}x=x,\forall{x\in{X_0}}$;
		\item [$\left(\mathrm{ii}\right)$] $\lim_{\lambda\to+\infty}{R_\lambda\left(A\right)}x=0,\forall{x\in{X}}$;
		\item [$\left(\mathrm{iii}\right)$] $\overline{D\left(A_0\right)}=X_0$.
	\end{itemize}
\end{lemma}
\begin{lemma}\cite[Theorem 5.3]{APazy1983}\label{lem2.4}
	A linear operator $A:D\left(A\right)\subset{X}\to{X}$ is an infinitesimal generator of a $C_0$-semigroup ${\left(T\left(t\right)\right)}_{t\geq0}$ satisfying ${\left\Vert{T\left(t\right)}\right\Vert}_{\mathcal{L}\left(X\right)}\leq{Me^{\vartheta{t}}}$, if and only if
	\begin{itemize}
		\item [$\left(\mathrm{i}\right)$] $A$ is densely defined in $X$;
		\item [$\left(\mathrm{ii}\right)$] $A$ is a Hille-Yosida operator.
	\end{itemize}
\end{lemma}
Then according to Assumption \ref{as2.1} and Lemma \ref{lem2.4}, we infer that $A_0$ generates a $C_0$-semigroup on $X_0$ denoted by ${\left(T\left(t\right)\right)}_{t\geq0}$, which  can also be denoted by $\left(e^{A_0t}\right)_{t\geq0}$. 
Now we give the definition of integrated semigroups.
\begin{definition}
	Let $\left(X,\left\Vert\cdot\right\Vert\right)$ be a Banach space. A family of bounded linear operators $\{S\left(t\right)\}_{t\geq0}$ on $X$ is called an integrated semigroup if
	\begin{itemize}
		\item[$\left(\mathrm{i}\right)$]$S\left(0\right)=0$;
		\item[$\left(\mathrm{ii}\right)$]the map $t\to{S\left(t\right)}x$ is continuous on $[0,+\infty)$ for each $x\in{X}$;
		\item[$\left(\mathrm{iii}\right)$]$S\left(t\right)$ satisfies
		\begin{equation*}
			S\left(s\right)S\left(t\right)=\int_{0}^{s}\left(S\left(r+t\right)-S\left(r\right)\right)dr, \forall{s,t\geq0}.
		\end{equation*}
	\end{itemize}
\end{definition}
An integrated semigroup $\{S\left(t\right)\}_{t\geq0}$ is said to be non-degenerate if whenever $S\left(t\right)x=0$, $\forall{t\geq0}$, then $x=0$. According to Thieme \cite{HS1990_}, we say that a linear operator $A:D\left(A\right)\subset{X}\to{X}$ is the generator of a non-degenerate integrated semigroup $\{S\left(t\right)\}_{t\geq0}$ on $X$ if and only if 
\begin{equation*}
	x\in{D\left(A\right)},~y=Ax\Leftrightarrow{S\left(t\right)x-tx=\int_0^t{S\left(s\right)yds}}, ~\forall{t\geq0}.
\end{equation*}
\begin{lemma}\cite[Proposition 2.5]{PM2007}\label{lem2.6}
	Let Assumption \ref{as2.1} be satisfied. Then $A$ generate a uniquely determined  non-degenerate integrated semigroup $\{S\left(t\right)\}_{t\geq0}$. Moreover, for each $x\in{X},t\geq0$ and each $\mu>\vartheta$, $S\left(t\right)x$ is given by 
	\begin{equation}
		S\left(t\right)x=\mu\int_0^tT\left(s\right){R_\mu\left(A\right)}ds+\left(I-T\left(t\right)\right){R_\mu\left(A\right)}x.
	\end{equation}
	Also, the map $t\to{S\left(t\right)}x$ is continuously differentiable if and only if $x\in{X_0}$ and 
	\begin{equation*}
		\frac{dS\left(t\right)x}{dt}=T\left(t\right)x, ~\forall{t\geq0},\ \forall{x\in{X_0}}. 
	\end{equation*}
\end{lemma}
From Lemma \ref{lem2.6}, we can deduce that for $\mu>\vartheta$, $S\left(t\right)$ commutes with ${\left(\mu{I}-A\right)}^{-1}$ and 
\begin{equation*}
	S\left(t\right)x=\int_0^tT\left(s\right)xds, ~\forall{x\in{X_0}}.
\end{equation*}
Hence, for each $x\in{X_0},~t\geq0$ and each $\mu>\vartheta$, 
\begin{equation}\label{eq2.2}
	{R_\mu\left(A\right)}S\left(t\right)x=S\left(t\right){R_\mu\left(A\right)}x=\int_0^tT\left(s\right){R_\mu\left(A\right)}xds.
\end{equation}
In addition, by Arendt \cite[Lemma 3.2.2 b),d)]{WA2001}, $S\left(t\right)x\in{X_0}$ for $x\in{X}$. Thus by \eqref{eq2.2} and Lemma \ref{lem2.3}(i), we have for each $x\in{X_0},~t\geq0$ and each $\mu>\vartheta$,
\begin{equation*}
	S\left(t\right)x=\lim_{\mu\to+\infty}\int_0^t{T\left(s\right)}\mu{R_\mu\left(A\right)}xds.
\end{equation*}
\subsection{Integrated solutions}
We need to find the integrated solution(or mild solution) of \eqref{eq1.1}. Recall the non-homongeneous Cauchy problem discussed in \cite{WA2001,PM2007,PM2009,PM_2009,Magal2016}
\begin{equation}\label{eq2.3}
	\frac{du}{dt}=Au\left(t\right)+f\left(t\right),~t\geq0,~u\left(0\right)=x\in{X_0}.
\end{equation}
with $f\in{L^1\left(\left(0,\tau_0\right),X\right)}$. If $A$ is a generator of a $C_0$-semigroup ${\left(T\left(t\right)\right)}_{t\geq0}$, then the variation of constants formula indicates that \eqref{eq2.3} has a unique mild solution $u\left(t\right)$ given by
\begin{equation*}
	u\left(t\right)=T\left(t\right)x+\int_s^t{T\left(t-s\right)}f\left(s\right)ds.
\end{equation*}
Now we consider \eqref{eq2.3} under the assumptions we set up for $A$ in this article.   
\begin{lemma}\cite[Lemma 2.6]{PM2007}\label{lem2.7}
	Let Assumption \ref{as2.1} be satisfied and let $\tau_0>0$ be fixed. Denote for each $f\in{C^1\left(\left[0,\tau_0\right],X\right)}$, 
	\begin{equation*}
		\left(S*f\right)\left(t\right)=\int_0^t{S\left(s\right)f\left(t-s\right)}ds,~\forall{t}\in\left[0,\tau_0\right].
	\end{equation*}
	Then we have the following:
	\begin{itemize}
		\item[$\left(\mathrm{i}\right)$]The map $t\to{\left(S*f\right)\left(t\right)}$ is continuously differentiable on $\left[0,\tau_0\right]$;
		\item[$\left(\mathrm{ii}\right)$]$\left(S*f\right)\left(t\right)\in{D\left(A\right)},~\forall{t\in\left[0,\tau_0\right]}$;
		\item[$\left(\mathrm{iii}\right)$]if we set $u\left(t\right)=\frac{d}{dt}\left(S*f\right)\left(t\right)$, then
		\begin{equation*}
			u\left(t\right)=A\int_0^t{u\left(s\right)}ds+\int_0^t{f\left(s\right)}ds, ~\forall{t\in\left[0,\tau_0\right]}.
		\end{equation*}
		\item[$\left(\mathrm{iv}\right)$]For each $\lambda>\vartheta$, and each $t\in\left[0,\tau_0\right]$, we have
		\begin{equation*}
			{R_\lambda\left(A\right)}\frac{d}{dt}\left(S*f\right)\left(t\right)=\int_0^tT\left(t-s\right){R_\lambda\left(A\right)}f\left(s\right)ds.
		\end{equation*}
	\end{itemize}
\end{lemma}
\begin{definition}
	A continuous map $u\in{C\left(\left[0,\tau_0\right],X\right)}$ is called an integrated solution(or mild solution) of \eqref{eq2.3} if and only if
	\begin{equation}\label{eq2.4}
		\int_0^tu\left(s\right)ds\in{D\left(A\right)}, ~\forall{t\in\left[0,\tau_0\right]}.
	\end{equation}
	and
	\begin{equation*}
		u\left(t\right)=x+A\int_0^t{u\left(s\right)}ds+\int_0^t{f\left(s\right)}ds, ~\forall{t\in\left[0,\tau_0\right]}.
	\end{equation*}
\end{definition}
We assume additionally $A$ is a closed operator, from \eqref{eq2.4} we know that if $u$ is an integrated solution of \eqref{eq2.4}, then $u\left(t\right)\in{X_0}$, $\forall{t\in\left[0,\tau_0
	\right]}$. And we need the following assumption:
\begin{assumption}\label{as2.2}
	Assume that there exists a non-decreasing map $\delta:\left[0,\tau_0\right]\to[0,+\infty)$ satisfying
	\begin{equation*}
		\lim_{t\to0^+}\delta\left(t\right)=0.
	\end{equation*}	
	such that for each $f\in{C\left(\left[0,\tau_0\right];X\right)}$,
	\begin{equation*}
		\left\Vert\frac{d}{dt}\left(S*f\right)\left(t\right)\right\Vert\leq\delta\left(t\right)\underset{s\in[0,t]}{\sup}\left\Vert{f\left(s\right)}\right\Vert,~\forall{t\in\left[0,\tau_0\right]}.
	\end{equation*}
\end{assumption}

\begin{lemma}\cite[Corollary 2.12]{PM2007}
	Let Assumption \ref{as2.1} and Assumption \ref{as2.2} be satisfied. Then for each $x\in{X_0}$ and each $f\in{L^1\left(\left(0,\tau_0\right),X\right)}$, \eqref{eq2.3} has a unique integrated solution $u\in{C\left(\left[0,\tau_0\right],X_0\right)}$ given by
	\begin{equation*}
		u\left(t\right)=T\left(t\right)x+\frac{d}{dt}\left(S*f\right)\left(t\right),~\forall{t\in\left[0,\tau_0\right]}.
	\end{equation*}
\end{lemma}
Since $\frac{d}{dt}\left(S*f\right)\left(t\right)\in{X_0}$ for all $t\in\left[0,\tau_0\right]$, denote $\frac{d}{dt}\left(S*f\right)\left(t\right)$ as $\left(S\diamond{f}\right)\left(t\right)$, by Lemma \ref{lem2.7}(iv) and Lemma \ref{lem2.3}(i), we have 
\begin{equation*}
	\left(S\diamond{f}\right)\left(t\right)=\lim_{\lambda\to+\infty}\int_0^tT\left(t-s\right)\lambda{R_\lambda\left(A\right)}f\left(s\right)ds, ~\lambda>\vartheta.
\end{equation*}

Similarly, we could define the integrated solution of \eqref{eq1.1}.
\begin{definition}\label{def2.10}
	For $\tau\in\mathbb{R}$, if $u\in{C\left([\tau,+\infty),L^2\left(\Omega,X\right)\right)}$ is a $\mathcal{F}_t$-progressively measurable process, then it's called an integrated solution(or mild solution) of \eqref{eq1.1} if and only if
	\begin{equation*}
		\int_\tau^tu\left(s\right)ds\in{D\left(A\right)}, ~\forall{t\in[\tau,+\infty)}.
	\end{equation*}    
	and
	\begin{equation*}
		u\left(t\right)=u_0+A\int_\tau^t{u\left(s\right)}ds+\int_\tau^t{F\left(u\left(s\right)\right)}ds+\int_\tau^t{\sigma\left(u\left(s\right)\right)}dW\left(s\right), ~\forall{t\in[\tau,+\infty)}.
	\end{equation*}
\end{definition}
Since we shall consider \eqref{eq1.1} with $t\in\mathbb{R}$, we give the following modified assumption 
\begin{assumption}\label{as2.3}
	Assume that there exists a non-decreasing map $\delta:\mathbb{R}\to[0,+\infty)$ satisfying
	\begin{equation*}
		\lim_{t\to0}\delta\left(t\right)=0.
	\end{equation*}	
	such that for each $f\in{C\left(\mathbb{R};X\right)}$,
	\begin{equation*}
		\left\Vert\left(S\diamond{f}\right)\left(t\right)\right\Vert\leq\delta\left(t\right)\underset{s\in[0,t]}{\sup}\left\Vert{f\left(s\right)}\right\Vert,~\forall{t\in\mathbb{R}}.
	\end{equation*}
\end{assumption}

\begin{lemma}\label{lem0}\cite[Proposition 2.13]{PM_2009}
	Let Assumption \ref{as2.1} and Assumption \ref{as2.3} be satisfied. Then for $\kappa>\vartheta$, there exists $C_\kappa>0$ such that $f\in{C\left(\mathbb{R};X\right)}$ and
	\begin{equation*}
		\left\Vert\left(S\diamond{f}\right)\left(t\right)\right\Vert\leq{C_\kappa}\underset{s\in[0,t]}{\sup}e^{\kappa\left(t-s\right)}\left\Vert{f\left(t\right)}\right\Vert,~\forall{t\in\mathbb{R}}.
	\end{equation*} 
	Moreover, for each $\varepsilon>0$, if $\rho_\varepsilon>0$ satisfying $M\delta\left(\rho_\varepsilon\right)\leq\varepsilon$, it holds
	\begin{equation*}
		C_\kappa=\frac{2\varepsilon{\max\left(1,e^{-\kappa\rho_\varepsilon}\right)}}{1-e^{\left(\vartheta-\kappa\right)\rho_\varepsilon}}. 
	\end{equation*}
\end{lemma}

We now are ready to present the existence of solution to \eqref{eq1.1}.

\begin{theorem}\label{thm2.12}
	Let Assumption\ref{as1}, Assumption \ref{as2.1} and Assumption \ref{as2.3} be satisfied. Given $\tau\in\mathbb{R}$ and $u_0\in{L^2\left(\Omega,\mathcal{F}_\tau;X_0\right)}$, then \eqref{eq1.1} possesses a unique integrated solution in $L^2\left(\Omega,X_0\right)$ on $[\tau,+\infty)$ given by
	\begin{equation}\label{eq2.5}
		\begin{split}
			u\left(t,\tau,u_0\right)=&T\left(t-\tau\right)u_0+\lim_{\lambda\to+\infty}\int_\tau^tT\left(t-s\right)\lambda{R_\lambda\left(A\right)}F\left(u\left(s\right)\right)ds\\&+\lim_{\lambda\to+\infty}\int_\tau^tT\left(t-s\right)\lambda{R_\lambda\left(A\right)}\sigma\left(u\left(s\right)\right)dW\left(s\right).
		\end{split}
	\end{equation}
\end{theorem}
\begin{proof}
Note that	\eqref{eq2.5} is the so-called modified variation of constants formula. By applying the resolvent operator ${\left(\lambda{I}-A\right)}^{-1}$ which maps $X$ to $X_0$  and regarding that $D\left(A_0\right)={\left(\lambda{I}-A\right)}^{-1}X_0$, we could turn \eqref{eq1.1} to a evolution equation on $X_0$ on which $A_0$ generates a $C_0$-semigroup as follows
	\begin{equation}\label{eq2.6}
		d{R_\lambda\left(A\right)}u\left(t\right)=A_0{R_\lambda\left(A\right)}u\left(t\right)dt+{R_\lambda\left(A\right)}F\left(u\left(t\right)\right)dt+{R_\lambda\left(A\right)}\sigma\left(u\left(t\right)\right)dW\left(t\right).  
	\end{equation}
	Then by the variation of constants formula, we could obtain the mild solution of \eqref{eq2.6} given by
	\begin{equation*}
		\begin{split}
			{R_\lambda\left(A\right)}u\left(t\right)=&T\left(t-\tau\right){R_\lambda\left(A\right)}u_0+\int_\tau^tT\left(t-s\right){R_\lambda\left(A\right)}F\left(u\left(s\right)\right)ds\\&+\int_\tau^tT\left(t-s\right){R_\lambda\left(A\right)}\sigma\left(u\left(s\right)\right)dW\left(s\right).
		\end{split}
	\end{equation*}
	By Lemma \ref{lem2.3}(i), we obtain \eqref{eq2.5}. Rewrite \eqref{eq2.5} as
	\begin{equation}\label{eq2.8}
		u\left(t\right)=T\left(t-\tau\right)u_0+\left(S\diamond{\left(F\left(u\right)+\sigma\left(u\right)dW\right)}\right)\left(s\right).
	\end{equation}
	By Lemma \ref{lem2.7}(iii) and Definition \ref{def2.10}, \eqref{eq2.8} is an integrated solution of \eqref{eq1.1}.  For given $\kappa_0>\vartheta$, we denote by $\mathcal{C}_{\kappa_0,\tau}$ the space of all $X_0$-valued $\mathcal{F}_t$-progressively measurable processes $f\left(t\right), t\in[\tau,+\infty)$ such that $f:[\tau,+\infty)\to{L^2\left(\Omega,X_0\right)}$ is continuous and \begin{equation*}
		\underset{t\geq\tau}{\sup}\left(e^{-\kappa{t}}{\left\Vert{f}\left(t\right)\right\Vert}_{L^2\left(\Omega,X_0\right)}\right)<\infty,
	\end{equation*}
	with norm
	\begin{equation*}
		{\left\Vert{f}\left(t\right)\right\Vert}_{\mathcal{C}_{\kappa_0,\tau}}=\underset{t\geq\tau}{\sup}\left(e^{-\kappa_0{t}}{\left\Vert{f}\left(t\right)\right\Vert}_{L^2\left(\Omega,X_0\right)}\right),~\forall{f\in{\mathcal{C}_{\kappa_0,\tau}}}.
	\end{equation*}
	Denote $\mathcal{G}_{u_0}\left(u\right)\left(t\right)$ the right side of \eqref{eq2.5}, i.e.
	\begin{equation*}\begin{split}
			\mathcal{G}_{u_0}\left(u\right)\left(t\right)=&T\left(t-\tau\right)u_0+\lim_{\lambda\to+\infty}\int_\tau^tT\left(t-s\right)\lambda{R_\lambda\left(A\right)}F\left(u\left(s\right)\right)ds\\&+\lim_{\lambda\to+\infty}\int_\tau^tT\left(t-s\right)\lambda{R_\lambda\left(A\right)}\sigma\left(u\left(s\right)\right)dW\left(s\right).
		\end{split}
	\end{equation*}
	For $u_1,u_2\in{C\left([\tau,+\infty),L^2\left(\Omega,X_0\right)\right)}$ be two $X_0$-valued $\mathcal{F}_t$-progressively measurable processes, we have, for $t\geq\tau$,
	\begin{equation*}\begin{split}
			&\left\Vert{\mathcal{G}_{u_0}\left(u_1\right)\left(t\right)-\mathcal{G}_{u_0}\left(u_2\right)\left(t\right)}\right\Vert_{L^2\left(\Omega,X_0\right)}\\&~\leq\lim_{\lambda\to+\infty}\int_0^{t-\tau}\left\Vert{T\left(t-g\right)\lambda{R_\lambda\left(A\right)}\left(F\left(u_1\left(g\right)\right)-F\left(u_2\left(g\right)\right)\right)}\right\Vert_{L^2\left(\Omega,X_0\right)}dl\\&~~+\left\Vert\lim_{\lambda\to+\infty}\int_0^{t-\tau}T\left(t-g\right)\lambda{R_\lambda\left(A\right)}\left(\sigma\left(u_1\left(g\right)\right)-\sigma\left(u_2\left(g\right)\right)\right)dW\left(l\right)\right\Vert_{L^2\left(\Omega,X_0\right)}
		\end{split}
	\end{equation*}
	where $g=l+\tau$. By Lemma \ref{lem0}, \eqref{eq1} and \eqref{eq2}, for $\varepsilon>0$, $\rho_\varepsilon>0$ with $M\delta\left(\rho_\varepsilon\right)\leq\varepsilon$ and $\kappa_0>\vartheta$, there exists a constant
	\begin{equation*}
		C_{\kappa_0}=\frac{2\varepsilon{\max\left(1,e^{-\kappa_0\rho_\varepsilon}\right)}}{1-e^{\left(\vartheta-\kappa_0\right)\rho_\varepsilon}},
	\end{equation*}
	such that
	\begin{equation}\label{eq2.50}\begin{split}
			&e^{-\kappa_0{t}}\left\Vert{\mathcal{G}_{u_0}\left(u_1\right)\left(t\right)-\mathcal{G}_{u_0}\left(u_2\right)\left(t\right)}\right\Vert_{L^2\left(\Omega,X_0\right)}\\&~\leq{L_1}C_{\kappa_0}\underset{l\in\left[0,t-\tau\right]}{\sup}e^{-\kappa_0g}\left\Vert{u_1\left(g\right)-u_2\left(g\right)}\right\Vert_{L^2\left(\Omega,X_0\right)}\\&~~+{C_{\kappa_0}}\underset{l\in[0,t-\tau]}{\sup}e^{-\kappa_0g}\mathbb{E}\left[{\left\Vert\sigma\left(u_1\left(g\right)\right)-\sigma\left(u_2\left(g\right)\right)\right\Vert}^2_{\mathcal{L}_2\left(Y_0,X\right)}\right]^{\frac{1}{2}}\\&~\leq{\left(L_1+L_2\right)}C_{\kappa_0}\left\Vert{u_1-u_2}\right\Vert_{\mathcal{C}_{\kappa_0,\tau}}.\end{split}	
	\end{equation}
	By \eqref{eq2.50}, we get that $\mathcal{G}_{u_0}\left(\cdot\right)\left(t\right):{\mathcal{C}_{\kappa_0,\tau}}\to{\mathcal{C}_{\kappa_0,\tau}}$ is well-defined. Let $\varepsilon>0$ such that $\varepsilon\max\left(\left(L_1+L_2\right),1\right)<1/8$ and $\kappa>\max\left(0,\vartheta\right)$ such that,
	\begin{equation*}
		\frac{1}{1-e^{\left(\vartheta-\kappa_0\right)\rho_\varepsilon}}<2,~\forall{\kappa_0\geq\kappa}.
	\end{equation*}
	Then by \eqref{eq2.50}, we get that
	\begin{equation*}
		\left\Vert{\mathcal{G}_{u_0}\left(u_1\right)\left(t\right)-\mathcal{G}_{u_0}\left(u_2\right)\left(t\right)}\right\Vert_{\mathcal{C}_{\kappa_0,\tau}}\leq4\left(L_1+L_2\right)\varepsilon\left\Vert{u_1-u_2}\right\Vert_{\mathcal{C}_{\kappa_0,\tau}}\leq\frac{1}{2}\left\Vert{u_1-u_2}\right\Vert_{\mathcal{C}_{\kappa_0,\tau}}.
	\end{equation*}
	So $\mathcal{G}_{u_0}$ is a contraction. From the contraction mapping principle, there exists a unique integrated solution of \eqref{eq1.1} in $\mathcal{C}_{\kappa_0,\tau}$ for any given $\tau\in\mathbb{R}$.
	Thus the uniqueness of the solution is proved.
\end{proof}
For further details, please refer to \cite{WA2001,Magal2016,AN2020}. Now we can define a mean-square random dynamical system for \eqref{eq1.1}. Given $t\in\mathbb{R}^+$ and $\tau\in\mathbb{R}$, let $\Phi(t,\tau)$ be a mapping from ${L^2\left(\Omega,\mathcal{F}_\tau;X_0\right)}$ to ${L^2\left(\Omega,\mathcal{F}_{\tau+t};X_0\right)}$ given by
\begin{equation*}
	\Phi(t,\tau)\left(u_0\right)=u\left(t+\tau,\tau,u_0\right).
\end{equation*}
where $u_0\in{L^2\left(\Omega,\mathcal{F}_\tau;X_0\right)}$. Due to the uniqueness of solutions, for $\forall{t,s\geq0}$ and $\tau\in\mathbb{R}$,
\begin{equation*}
	\Phi\left(t+s,\tau\right)=\Phi\left(t,s+\tau\right)\circ\Phi\left(s,\tau\right).
\end{equation*}
As a consequence, $\Phi$ is a mean-square random dynamical system generated by solutions of \eqref{eq1.1} in ${L^2\left(\Omega,\mathcal{F};X\right)}$. Based on the forward solutions \eqref{eq2.5}, we could define the backward solutions on $(-\infty,\tau]$. Given $\tau\in\mathbb{R}$, an $X_0$-valued $\mathcal{F}_t$-progressively measurable process $\xi\left(t\right),t\in(-\infty,\tau]$, is called an integrated solution of \eqref{eq1.1} on $(-\infty,\tau]$ if $\xi\left(r\right)\in{L^2\left(\Omega,\mathcal{F}_r;X_0\right)}$ for all $r\in(-\infty,\tau]$ and $u\left(t,r,\xi\left(r\right)\right)=\xi\left(t\right)$ for all $r\leq{t}\leq{\tau}$, where $u\left(t,r,\xi\left(r\right)\right)=\xi\left(t\right)$ is the unique solution of \eqref{eq1.1} with initial value $\xi\left(r\right)$ at initial time $r$.  
\subsection{Exponential dichotomy}
\begin{assumption}\label{as2.5}
	We assume $T\left(t\right)$ satisfies the pseudo exponential dichotomy with exponents $\beta<\alpha$ and bound $K$. That's, there exists a continuous linear projection operator with finite rank $\Pi_{0u}\in\mathcal{L}\left(X_0\right)$ such that
	\begin{equation*}
		\Pi_{0u}T\left(t\right)=T\left(t\right)\Pi_{0u},~\forall{t\in\mathbb{R}}.
	\end{equation*}
	Denote $\Pi_{0s}=I_{X_0}-\Pi_{0u}$, $X_{0u}=\Pi_{0u}{X_0}$ and $X_{0s}=\Pi_{0s}{X_0}$, it holds
	\begin{equation}\label{eq2.9}
		\left\Vert{T_{A_{0u}}\left(t\right)\Pi_{0u}}\right\Vert\leq{Ke^{\alpha{t}}},~\forall{t\leq0}.
	\end{equation}
	\begin{equation}\label{eq2.10}
		\left\Vert{T_{A_{0s}}\left(t\right)\Pi_{0s}}\right\Vert\leq{Ke^{\beta{t}}},~\forall{t\geq0}.
	\end{equation}
	where $A_{0p}=A_0|_{X_{0p}}$ and ${T_{A_{0p}}\left(t\right)}$ is the $C_0$-semigroup generated by $A_{0p}$, ~$\forall{p\in\{u,s\}}$.
\end{assumption}
\begin{remark}\label{rem}
Assume $\alpha>\gamma\ge0\ge-\gamma>-\beta$. If $X_0$ is an infinite dimensional space and the spectrum of $A_0$ satisfies $\sigma\left(A_0\right)=\sigma^{0s}\cup\sigma^{0c}\cup\sigma^{0u}$, where $\sigma^{0s}={\left\{\lambda\in\sigma\left(A_0\right):\text{Re}\left(\lambda\right)\leq-\beta\right\}}$, $\sigma^{0c}={\left\{\lambda\in\sigma\left(A_0\right):\left|\text{Re}\left(\lambda\right)\right|\leq\gamma\right\}}$, $\sigma^{0u}={\left\{\lambda\in\sigma\left(A_0\right):\text{Re}\left(\lambda\right)\geq\alpha\right\}}$, and $A_0$ generates a strong continuous semigroup, then the exponential trichotomy holds, too (see \cite[page 267]{TG1993}. And if the spectrum of $A_0$ satisfies $\sigma\left(A_0\right)=\sigma^{0s}\cup\sigma^{0u}$, where $\sigma^{0s}$ and $\sigma^{0u}$ are the same as above, then the exponential dichotomy holds. Since $\sigma\left(A\right)=\sigma\left(A_0\right)$, the spectrum of $A$ can be split as $A_0$. 
\end{remark}

Since we only require the Hille-Yosida condition on $X_0$, one needs to extend the mentioned projections from $X_0$ to $X$ and outline the associated statement below (see \cite[Proposition 3.5]{PM_2009}.
\begin{lemma}\label{lem2.12}
	Let $\Pi_0:X_0\to{X_0}$  be a bounded linear projection satisfying
	\begin{equation*}
		\Pi_0{R_\lambda\left(A_0\right)}={R_\lambda\left(A_0\right)}\Pi_0,~\forall{\lambda>\vartheta},
	\end{equation*}
	and
	\begin{equation*}
		\Pi_0X_0\subset{D\left(A_0\right)}, ~\text{and}~{A_0|}_{\Pi_0X_0}~\text{is bounded}.
	\end{equation*}
	Then there exists a unique bounded projection $\Pi:X\to{X}$ such that
	\begin{itemize}
		\item [$\left(\mathrm{i}\right)$]${\Pi|}_{X_0}=\Pi_0$;
		\item [$\left(\mathrm{ii}\right)$]$\Pi\left(X\right)\subset{X_0}$;
		\item [$\left(\mathrm{iii}\right)$]$\Pi{R_\lambda\left(A\right)}={R_\lambda\left(A\right)}\Pi$,~$\forall{\lambda>\vartheta}$.
	\end{itemize}
	Moreover, for each $x\in{X}$, we have
	\begin{equation*}
		\Pi{x}=\lim_{\lambda\to+\infty}\Pi_0\lambda{R_\lambda\left(A\right)}x=\lim_{h\to0^+}\frac1{h}\Pi_0S\left(h\right)x,
	\end{equation*}
	where $S$ is the integrated semigroup generated by $A$.
\end{lemma}
We use Lemma \ref{lem2.12} to state the decomposition on $X$. Denote $\Pi_u:X\to{X}$ the unique extension of $\Pi_{0u}$ and let $\Pi_s=I_X-\Pi_u$. Then we have for each ${p\in\left\{u,s\right\}}$,
\begin{equation*}
	\Pi_p{R_\lambda\left(A\right)}={R_\lambda\left(A\right)}\Pi_p,~\forall{\lambda>\vartheta},
\end{equation*} 
and
\begin{equation*}
	\Pi_p\left(X_0\right)\subset{X_0}.
\end{equation*}
Set
\begin{equation*}
	X_{0p}=\Pi_p\left(X_0\right),~X_p=\Pi_p\left(X\right),~\text{and}~A_p=A|_{X_p}.
\end{equation*}
Then we have $X_u=X_{0u}$. Moreover, we have
\begin{equation*}
	X_0=X_{0s}\oplus{X_{0u}},~\text{and}~X=X_{s}\oplus{X_{u}}.
\end{equation*}
Herein, $X_{0u}$ and $X_{0s}$ are called unstable subspace and stable subspace of $X_0$ respectively.  
For more details, see \cite[page 22]{PM_2009}.
Then we give the definition of mean-square random invariant manifolds and mean-square random invariant sets.
\begin{definition}
	Let $\mathcal{M}=\left\{\mathcal{M}\left(\tau\right)\subseteq{L^2\left(\Omega,\mathcal{F}_\tau;X_0\right)}:\tau\in\mathbb{R}\right\}$ be a famliy of subsets of the space $L^2\left(\Omega,\mathcal{F};X_0\right)$.
	\begin{itemize}
		\item [$\left(\mathrm{i}\right)$] $\mathcal{M}$ is called a mean-square random unstable invariant set of $\Phi$ in $X_0$ if $\mathcal{M}$ is invariant under $\Phi$ and there is a family $h^u=\left\{h^u\left(\cdot,\tau\right):\tau\in\mathbb{R}\right\}$ of mappings such that for every $\tau\in\mathbb{R}$, $h^u\left(x,\tau\right):D\left(h^u\left(x,\tau\right)\right)\subseteq{L^2\left(\Omega,\mathcal{F}_\tau;X_{0u}\right)}\to{L^2\left(\Omega,\mathcal{F}_\tau;X_{0s}\right)}$ is Lipschitz continuous and 
		\begin{equation*}
			\mathcal{M}\left(\tau\right)=\left\{x+h^u\left(x,\tau\right):x\in{D\left(h^u\left(x,\tau\right)\right)}\right\};
		\end{equation*}
		\item [$\left(\mathrm{ii}\right)$] $\mathcal{M}$ is called a mean-square random unstable invariant manifold of $\Phi$ in $X_0$ if $\mathcal{M}$ is invariant under $\Phi$ and there is a family $h^u=\left\{h^u\left(\cdot,\tau\right):\tau\in\mathbb{R}\right\}$ of mappings such that for every $\tau\in\mathbb{R}$, $h^u\left(x,\tau\right):{L^2\left(\Omega,\mathcal{F}_\tau;X_{0u}\right)}\to{L^2\left(\Omega,\mathcal{F}_\tau;X_{0s}\right)}$ is Lipschitz continuous and 
		\begin{equation*}
			\mathcal{M}\left(\tau\right)=\left\{x+h^u\left(x,\tau\right):x\in{L^2\left(\Omega,\mathcal{F}_\tau;X_{0u}\right)}\right\};
		\end{equation*}
		\item [$\left(\mathrm{iii}\right)$] $\mathcal{M}$ is called a mean-square random stable invariant set of $\Phi$ in $X_0$ if $\mathcal{M}$ is invariant under $\Phi$ and there is a family $h^s=\left\{h^s\left(\cdot,\tau\right):\tau\in\mathbb{R}\right\}$ of mappings such that for every $\tau\in\mathbb{R}$, $h^s\left(x,\tau\right):D\left(h^s\left(x,\tau\right)\right)\subseteq{L^2\left(\Omega,\mathcal{F}_\tau;X_{0s}\right)}\to{L^2\left(\Omega,\mathcal{F}_\tau;X_{0u}\right)}$ is Lipschitz continuous and 
		\begin{equation*}
			\mathcal{M}\left(\tau\right)=\left\{x+h^s\left(x,\tau\right):x\in{D\left(h^s\left(x,\tau\right)\right)}\right\};
		\end{equation*}
		\item [$\left(\mathrm{iv}\right)$] $\mathcal{M}$ is called a mean-square random stable invariant set of $\Phi$ in $X_0$ if $\mathcal{M}$ is invariant under $\Phi$ and there is a family $h^s=\left\{h^s\left(\cdot,\tau\right):\tau\in\mathbb{R}\right\}$ of mappings such that for every $\tau\in\mathbb{R}$, $h^s\left(x,\tau\right):{L^2\left(\Omega,\mathcal{F}_\tau;X_{0s}\right)}\to{L^2\left(\Omega,\mathcal{F}_\tau;X_{0u}\right)}$ is Lipschitz continuous and 
		\begin{equation*}
			\mathcal{M}\left(\tau\right)=\left\{x+h^s\left(x,\tau\right):x\in{L^2\left(\Omega,\mathcal{F}_\tau;X_{0s}\right)}\right\}.
		\end{equation*}
	\end{itemize}
\end{definition}

\section{Mean-square random unstable invariant manifolds}\label{sec3}
This section is dedicated to the existence of mean-square random unstable invariant manifolds of \eqref{eq1.1} through the Lyapunov-Perron method. Given $\tau\in\mathbb{R}$, denote by $\mathcal{C}^-_{\tau}$ the space of all $X_0$-valued $\mathcal{F}_t$-progressively measurable  processes $\xi\left(t\right), t\in(-\infty,\tau]$ such that $\xi:(-\infty,\tau]\to{L^2\left(\Omega,X_0\right)}$ is continuous and \begin{equation}\label{eq100}
	\underset{t\leq\tau}{\sup}\left(e^{-\gamma{t}}{\left\Vert\xi\left(t\right)\right\Vert}_{L^2\left(\Omega,X_0\right)}\right)<\infty,
\end{equation}
with norm
\begin{equation*}
	{\left\Vert\xi\left(t\right)\right\Vert}_{\mathcal{C}^-_\tau}=\underset{t\leq\tau}{\sup}\left(e^{-\gamma{t}}{\left\Vert\xi\left(t\right)\right\Vert}_{L^2\left(\Omega,X_0\right)}\right),~\forall{\xi\in{\mathcal{C}^-_{\tau}}}.
\end{equation*}
where $\gamma\in\left(\beta,\alpha\right)$. Notice that ${\mathcal{C}^-_\tau}$ is a Banach space.
To construct mean-square random unstable invariant manifolds of \eqref{eq1.1}, we need to find all solutions of \eqref{eq1.1} in the space ${\mathcal{C}^-_\tau}$.
\begin{lemma}\label{lem3.1}
	Let Assumption \ref{as1}, Assumption \ref{as2.1}, Assumption \ref{as2.3} and Assumption \ref{as2.5} be satisfied. $\xi\in{\mathcal{C}^-_\tau}$ for some $\tau\in\mathbb{R}$. Then $\xi$ is an integrated solution of \eqref{eq1.1} on $(\infty,\tau]$ if and only if there exists $x\in{L^2\left(\Omega,\mathcal{F}_\tau;X_{0u}\right)}$ such that for all $t\leq\tau$,
	\begin{equation}\label{eq3.2}
		\begin{split}
			\xi\left(t\right)=&T_{A_{0u}}\left(t-\tau\right)x-\int_t^\tau{T_{A_{0u}}\left(t-r\right)}\Pi_uF\left(\xi\left(r\right)\right)dr\\&-\int_t^\tau{T_{A_{0u}}\left(t-r\right)}\Pi_u\sigma\left(\xi\left(r\right)\right)dW\left(r\right)\\&+\lim_{\lambda\to+\infty}\int_{-\infty}^tT_{A_{0s}}\left(t-r\right)\lambda{R_\lambda\left(A_s\right)}\Pi_sF\left(\xi\left(r\right)\right)dr\\&+\lim_{\lambda\to+\infty}\int_{-\infty}^tT_{A_{0s}}\left(t-r\right)\lambda{R_\lambda\left(A_s\right)}\Pi_s\sigma\left(\xi\left(r\right)\right)dW\left(r\right).
		\end{split}    
	\end{equation}
	in $L^2\left(\Omega,X_0\right)$.
\end{lemma}
\begin{proof}
	\noindent {\bf Step 1: necessity.} Suppose $\xi\in{\mathcal{C}^-_\tau}$ with $\tau\in\mathbb{R}$ is an integrated solution of \eqref{eq1.1}. We prove that \eqref{eq3.2} holds ture. Since $\xi\in{\mathcal{C}^-_\tau}$, by \eqref{eq2.9} and Assumption \ref{as1}, the first two integrals on the right side of \eqref{eq2.3} are well-defined in $L^2\left(\Omega,X_0\right)$. For the third integral, according to Lemma \ref{lem0}, \eqref{eq_1} and \eqref{eq2.10}, for $\zeta\in\left(\beta,\gamma\right)$,
	\begin{equation}\label{eq3.3}
		\begin{split}
			&\lim_{\lambda\to+\infty}\int_{-\infty}^t{\left\Vert{T_{A_{0s}}\left(t-r\right)\lambda{R_\lambda\left(A_s\right)}\Pi_sF\left(\xi\left(r\right)\right)}\right\Vert}_{L^2\left(\Omega,X_0\right)}{dr}\\&~~=\lim_{\lambda\to+\infty}\lim_{g\to{-\infty}}\int_0^{t-g}{\left\Vert{T_{A_{0s}}\left(t-g-l\right)\lambda{R_\lambda\left(A_s\right)}\Pi_sF\left(\xi\left(l+g\right)\right)}\right\Vert}_{L^2\left(\Omega,X_0\right)}{dl}\\&~~\leq\lim_{g\to{-\infty}}L_1C_\zeta\underset{l\in[0,t-g]}{\sup}e^{\zeta\left(t-g-l\right)}{\left\Vert\xi\left(l+g\right)\right\Vert}_{L^2\left(\Omega,X_0\right)}\\&~~=L_1C_\zeta\underset{q\in(-\infty,t]}{\sup}e^{\zeta\left(t-q	\right)}{\left\Vert\xi\left(q\right)\right\Vert}_{L^2\left(\Omega,X_0\right)}\\&~~\leq{L_1}C_\zeta\underset{q\in(-\infty,t]}{\sup}e^{\zeta\left(t-q\right)}e^{\gamma{q}}{\left\Vert\xi\left(q\right)\right\Vert}_{\mathcal{C}^-_\tau}\\&~~={L_1}C_\zeta\underset{q\in(-\infty,t]}{\sup}e^{\left(\gamma-\zeta\right)q}e^{\left(\zeta-\gamma\right)t}e^{\gamma{t}}{\left\Vert\xi\left(q\right)\right\Vert}_{\mathcal{C}^-_\tau}\\&~~={L_1}C_\zeta{e^{\gamma{t}}}{\left\Vert\xi\right\Vert}_{\mathcal{C}^-_\tau}<\infty.
		\end{split}
	\end{equation}
	By \eqref{eq3.3}, for $t\leq\tau$,
	\begin{equation*}
		\lim_{\lambda\to+\infty}\int_{-\infty}^t{T_{A_{0u}}\left(t-r\right)\lambda{R_\lambda\left(A_s\right)}\Pi_sF\left(\xi\left(r\right)\right)}{dr}
	\end{equation*}
	is well-defined in $L^2\left(\Omega,X_0\right)$. For the forth integral, according to Lemma \ref{lem0}, \eqref{eq_2} and \eqref{eq2.9}, for $\zeta\in\left(\beta,\gamma\right)$,
	\begin{equation}\label{eq3.4}
		\begin{split}
			&{\left\Vert\lim_{\lambda\to+\infty}\int_{-\infty}^tT_{A_{0s}}\left(t-r\right)\lambda{R_\lambda\left(A_s\right)}\Pi_s\sigma\left(\xi\left(r\right)\right)dW\left(r\right)\right\Vert}_{L^2\left(\Omega,X_0\right)}\\&~~=\mathbb{E}\left[\left(\lim_{\lambda\to+\infty}\lim_{g\to{-\infty}}\int_0^{t-g}T_{A_{0s}}\left(t-l-g\right)\lambda{R_\lambda\left(A_s\right)}\Pi_s{\sigma\left(\xi\left(l+g\right)\right)}dW\left(l\right)\right)^2\right]^{\frac{1}{2}}\\&~~\leq\lim_{g\to{-\infty}}C_\zeta\underset{l\in[0,t-g]}{\sup}e^{\zeta\left(t-g-l\right)}\mathbb{E}\left[{\left\Vert\sigma\left(\xi\left(l+g\right)\right)\right\Vert}^2_{\mathcal{L}_2\left(Y_0,X\right)}\right]^{\frac{1}{2}}\\&~~\leq{L_2}C_\zeta\underset{q\in(-\infty,t]}{\sup}e^{\zeta\left(t-q	\right)}e^{\gamma{q}}{\left\Vert\xi\left(q\right)\right\Vert}_{\mathcal{C}^-_\tau}
			\\&~~={L_2}C_\zeta\underset{q\in(-\infty,t]}{\sup}e^{\left(\gamma-\zeta\right)q}e^{\left(\zeta-\gamma\right)t}e^{\gamma{t}}{\left\Vert\xi\left(q\right)\right\Vert}_{\mathcal{C}^-_\tau}\\&~~={L_2}C_\zeta{e^{\gamma{t}}}{\left\Vert\xi\right\Vert}_{\mathcal{C}^-_\tau}\\&~~<\infty.
		\end{split}
	\end{equation}
	By \eqref{eq3.4}, for $t\leq\tau$,
	\begin{equation*}
		\lim_{\lambda\to+\infty}\int_{-\infty}^tT_{A_{0s}}\left(t-r\right)\lambda{R_\lambda\left(A_s\right)}\Pi_s\sigma\left(\xi\left(r\right)\right)dW\left(r\right).
	\end{equation*}
	is well-defined in $L^2\left(\Omega,X_0\right)$. 
	
	Since $\xi\in{\mathcal{C}^-_\tau}$ is a solution of \eqref{eq1.1} on $(-\infty,\tau]$, by \eqref{eq2.5}, for all $s\leq{t}\leq\tau$, 
	\begin{equation}\label{eq3.5}
		\begin{split}
			\xi\left(t\right)=&T\left(t-s\right)\xi\left(s\right)+\lim_{\lambda\to+\infty}\int_s^tT\left(t-r\right)\lambda{R_\lambda\left(A\right)}F\left(\xi\left(r\right)\right)dr\\&+\lim_{\lambda\to+\infty}\int_s^tT\left(t-r\right)\lambda{R_\lambda\left(A\right)}\sigma\left(\xi\left(r\right)\right)dW\left(r\right).
		\end{split}
	\end{equation}
	Then we apply projections $\Pi_{0u}$ and $\Pi_{0s}$ to \eqref{eq3.5}. Note that $\Pi_u:X\to{X}$ is a bounded linear projection satisfying
	\begin{equation*}
		\Pi_u{R_\lambda\left(A\right)}={R_\lambda\left(A\right)}\Pi_u,~\forall{\lambda>\vartheta}.
	\end{equation*}
	Also, $A|_{\Pi_u\left(X\right)}=A|_{\Pi_u\left(X\right)}=A|_{X_u}=A_u$ satisfies Assumption \ref{as2.1} in $\Pi_u\left(X\right)=X_u$. Moreover, we know that $\Pi_u$ has a finite rank and $\Pi_u\left(X\right)\subset{X_0}$, which further indicates that  $A_0|_{\Pi_u\left(X_0\right)}=A_0|_{X_{0u}}=A_{0u}$ and $A|_{\Pi_u\left(X\right)}=A|_{X_u}=A_u$. Thus if for each $x\in{X_0}$ and each $t\leq\tau$, then the map $t\to\Pi_u{\xi\left(t\right)}$ solves
	\begin{equation*}
		d\Pi_u{\xi\left(t\right)}=A_{0u}\Pi_u{\xi\left(t\right)}dt+\Pi_u{F\left(\xi\left(t\right)\right)}dt+\Pi_u{\sigma\left(\xi\left(t\right)\right)}dW\left(t\right)
	\end{equation*}
	in $\Pi_u\left(X_0\right)=X_{0u}$, then for all $s\leq{t}\leq\tau$, 
	\begin{equation}\label{eq3.6}
		\begin{split}
			\Pi_u{\xi\left(t\right)}=&T_{A_0|_{\Pi_s\left(X_0\right)}}\left(t-s\right)\Pi_u{\xi\left(s\right)}+\left(S_{{A|}_{\Pi_u\left(X\right)}}\diamond{\Pi_u{\left(F\left(\xi\right)+\sigma\left(\xi\right)dW\right)}}\right)\left(r\right)\\=&T_{A_{0s}}\left(t-s\right)\Pi_u{\xi\left(s\right)}+\left(S_{{A}_{u}}\diamond{\Pi_u{\left(F\left(\xi\right)+\sigma\left(\xi\right)dW\right)}}\right)\left(r\right)\\=&T_{A_{0u}}\left(t-s\right)\Pi_u{\xi\left(s\right)}+\lim_{\lambda\to+\infty}\int_s^tT_{A_{0u}}\left(t-r\right)\lambda{R_\lambda\left(A_u\right)}\Pi_u{F\left(\xi\left(r\right)\right)}dr\\&+\lim_{\lambda\to+\infty}\int_s^tT_{A_{0u}}\left(t-r\right)\lambda{R_\lambda\left(A_u\right)}\Pi_u{\sigma\left(\xi\left(r\right)\right)}dW\left(r\right).
		\end{split}
	\end{equation}
	Since $\Pi_u$ is a extension of $\Pi_{0u}$ from $X_0$ to $X$ and $\Pi_u\left(X\right)\subset{X_0}$, by \eqref{eq3.6} and Lemma \ref{lem2.3}(i), we have
	\begin{equation}\label{eq3.7}
		\begin{split}
			\Pi_{0u}{\xi\left(t\right)}=&\Pi_u{\xi\left(t\right)}=T_{A_{0u}}\left(t-s\right)\Pi_{0u}{\xi\left(s\right)}+\int_s^tT_{A_{0u}}\left(t-r\right)\Pi_u{F\left(\xi\left(r\right)\right)}dr\\&+\int_s^tT_{A_{0u}}\left(t-r\right)\Pi_u{\sigma\left(\xi\left(r\right)\right)}dW\left(r\right).
		\end{split}    
	\end{equation}
	Through similar arguments, we have
	\begin{equation}\label{eq3.8}
		\begin{split}
			\Pi_{0s}{\xi\left(t\right)}=&T_{A_{0s}}\left(t-s\right)\Pi_{0s}{\xi\left(s\right)}+\lim_{\lambda\to+\infty}\int_s^tT_{A_{0s}}\left(t-r\right)\lambda{R_\lambda\left(A_s\right)}\Pi_s{F\left(\xi\left(r\right)\right)}dr\\&+\lim_{\lambda\to+\infty}\int_s^tT_{A_{0s}}\left(t-r\right)\lambda{R_\lambda\left(A_s\right)}\Pi_s{\sigma\left(\xi\left(r\right)\right)}dW\left(r\right).
		\end{split}
	\end{equation}
	In \eqref{eq3.7}, let $t=\tau$, 
	\begin{equation*}
		\begin{split}
			\Pi_{0u}{\xi\left(\tau\right)}=&T_{A_{0u}}\left(\tau-s\right)\Pi_{0u}{\xi\left(s\right)}+\int_s^\tau{T_{A_{0u}}}\left(\tau-r\right)\Pi_u{F\left(\xi\left(r\right)\right)}dr\\&+\int_s^\tau{T_{A_{0u}}}\left(\tau-r\right)\Pi_u{\sigma\left(\xi\left(r\right)\right)}dW\left(r\right).
		\end{split}  
	\end{equation*}
	Hence for all $s\leq{t}\leq\tau$, 
	\begin{equation}\label{eq3.9}
		\begin{split}
			T_{A_{0u}}\left(t-\tau\right)\Pi_{0u}\xi\left(\tau\right)=&T_{A_{0u}}\left(t-s\right)\Pi_{0u}\xi\left(s\right)+\int_s^\tau{T_{A_{0u}}}\left(t-r\right)\Pi_u{F\left(\xi\left(r\right)\right)}dr\\&+\int_s^\tau{T_{A_{0u}}}\left(t-r\right)\Pi_u{\sigma\left(\xi\left(r\right)\right)}dW\left(r\right).
		\end{split}    
	\end{equation}
	By \eqref{eq3.7} and \eqref{eq3.9}, 
	\begin{equation*}
		\begin{split}
			\Pi_{0u}\xi\left(t\right)=&T_{A_{0u}}\left(t-\tau\right)\Pi_{0u}\xi\left(\tau\right)-\int_s^\tau{T_{A_{0u}}}\left(t-r\right)\Pi_u{F\left(\xi\left(r\right)\right)}dr\\&-\int_s^\tau{T_{A_{0u}}}\left(t-r\right)\Pi_u{\sigma\left(\xi\left(r\right)\right)}dW\left(r\right)+\int_s^tT_{A_{0u}}\left(t-r\right)\Pi_u{F\left(\xi\left(r\right)\right)}dr\\& +\int_s^tT_{A_{0u}}\left(t-r\right)\Pi_u{\sigma\left(\xi\left(r\right)\right)}dW\left(r\right).
		\end{split}
	\end{equation*}
	Thus 
	\begin{equation}\label{eq3.10}
		\begin{split}
			\Pi_{0u}\xi\left(t\right)=&T_{A_{0u}}\left(t-\tau\right)\Pi_{0u}\xi\left(\tau\right)+\int_\tau^tT_{A_{0u}}\left(t-r\right)\Pi_u{F\left(\xi\left(r\right)\right)}dr\\&+\int_\tau^tT_{A_{0u}}\left(t-r\right)\Pi_u{\sigma\left(\xi\left(r\right)\right)}dW\left(r\right).
		\end{split}
	\end{equation}
	By \eqref{eq3.8} and \eqref{eq3.10}, for all $s\leq{t}\leq\tau$, we have
	\begin{equation}\label{eq3.11}
		\begin{split}
			\xi\left(t\right)=&T_{A_{0u}}\left(t-\tau\right)\Pi_{0u}\xi\left(\tau\right)+\int_\tau^tT_{A_{0u}}\left(t-r\right)\Pi_u{F\left(\xi\left(r\right)\right)}dr\\&+\int_\tau^tT_{A_{0u}}\left(t-r\right)\Pi_u{\sigma\left(\xi\left(r\right)\right)}dW\left(r\right)+T_{A_{0s}}\left(t-s\right)\Pi_{0s}{\xi\left(s\right)}\\&+\lim_{\lambda\to+\infty}\int_s^tT_{A_{0s}}\left(t-r\right)\lambda{R_\lambda\left(A_s\right)}\Pi_s{F\left(\xi\left(r\right)\right)}dr\\&+\lim_{\lambda\to+\infty}\int_s^tT_{A_{0s}}\left(t-r\right)\lambda{R_\lambda\left(A_s\right)}\Pi_s{\sigma\left(\xi\left(r\right)\right)}dW\left(r\right).
		\end{split}
	\end{equation}
	Since for all $s\leq{t}\leq\tau$, let $s\to{-\infty}$ in \eqref{eq3.11}, then by \eqref{eq2.10},
	\begin{equation*}\begin{split}
			{\left\Vert{T_{A_{0s}}\left(t-s\right)\Pi_{0s}{\xi\left(s\right)}}\right\Vert}_{L^2\left(\Omega,X_0\right)}&\leq{Ke^{\beta\left(t-s\right)}}{\left\Vert{\xi\left(s\right)}\right\Vert}_{L^2\left(\Omega,X_0\right)}\\&\leq{Ke^{\beta{t}}e^{\left(\gamma-\beta\right)s}}{\left\Vert{\xi\left(s\right)}\right\Vert}_{\mathcal{C}^-_\tau}\to0.
		\end{split}
	\end{equation*}
	Thus \eqref{eq3.2} is fulfilled as wished with $x=\Pi_{0u}\xi\left(\tau\right)$. 
	
	\noindent {\bf Step 2: sufficiency.} Suppose $\xi\in{\mathcal{C}^-_\tau}$ satisfies \eqref{eq3.2} for $t\leq\tau$. Since $\xi\in{\mathcal{C}^-_\tau}$, we know $\xi$ is an $X_0$-valued $\mathcal{F}_t$-progressively measurable process and $\xi\in{C\left((-\infty,\tau],L^2\left(\Omega,X_0\right)\right)}$. And from \eqref{eq3.2}, one can verify that for all $r\leq{t}\leq\tau$, \begin{equation}\label{eq3.12}
		\begin{split}
			\xi\left(t\right)=&T\left(t-r\right)\xi\left(r\right)+\lim_{\lambda\to+\infty}\int_r^tT\left(t-s\right)\lambda{R_\lambda\left(A\right)}F\left(\xi\left(s\right)\right)ds\\&+\lim_{\lambda\to+\infty}\int_r^tT\left(t-s\right)\lambda{R_\lambda\left(A\right)}\sigma\left(\xi\left(s\right)\right)dW\left(s\right).
		\end{split}
	\end{equation}
	From \eqref{eq3.12}, we get that for all $r\leq{t}\leq\tau$,
	\begin{equation*}
		\xi\left(t\right)=u\left(t,r,\xi\left(r\right)\right),
	\end{equation*}
	where $u$ is the integrated solution of \eqref{eq1.1} with initial condition $\xi\left(r\right)$ at initial time $r$. Then by the definition of integrated solutions of \eqref{eq1.1} on $(-\infty,\tau]$, $\xi$ is an integrated solution of \eqref{eq1.1} on $(-\infty,\tau]$.
\end{proof}
Now we are ready to construct mean-square unstable manifolds for \eqref{eq1.1}. By Lemma \ref{lem3.1}, if $\xi\in{\mathcal{C}^-_\tau}$ is an integrated solution of \eqref{eq1.1}, then $\xi$ must satisfy \eqref{eq3.2}. Take the conditional expectation $\mathbb{E}\left(\cdot|\mathcal{F}_t\right)$ on 
\eqref{eq3.2}, we get for $t\leq\tau$,
\begin{equation}\label{eq3.13}
	\begin{split}
		\xi\left(t\right)=&\mathbb{E}\left[T_{A_{0u}}\left(t-\tau\right)x-\int_t^\tau{T_{A_{0u}}\left(t-r\right)}\Pi_u{F}\left(\xi\left(r\right)\right)d\left(r\right)|\mathcal{F}_t\right]\\&+\lim_{\lambda\to+\infty}\int_{-\infty}^tT_{A_{0s}}\left(t-r\right)\lambda{R_\lambda\left(A_s\right)}\Pi_sF\left(\xi\left(r\right)\right)dr\\&+\lim_{\lambda\to+\infty}\int_{-\infty}^tT_{A_{0s}}\left(t-r\right)\lambda{R_\lambda\left(A_s\right)}\Pi_s\sigma\left(\xi\left(r\right)\right)dW\left(r\right).
	\end{split}
\end{equation}
in $L^2\left(\Omega,X_0\right)$ with $x\in{L^2\left(\Omega,\mathcal{F}_\tau;X_{0u}\right)}$.
For $\tau\in\mathbb{R}$ and $t\leq\tau$, set
\begin{equation*}
	\mathcal{M}^u\left(\tau\right)=\left\{\xi\left(\tau\right)\in{{L^2\left(\Omega,\mathcal{F}_\tau;X_0\right)}}:\xi\left(t\right)\in{\mathcal{C}^-_\tau} \text{and satisfies}~\eqref{eq3.13}\right\}.
\end{equation*}
We below prove that $ \mathcal{M}^u\left(\tau\right)$ is  invariant under $\Phi$ and it is given by a graph of  Lipschitz function.
\begin{theorem}\label{thm3.2}
	Assume Assumption \ref{as1}, Assumption \ref{as2.1}, Assumption \ref{as2.3} and Assumption \ref{as2.5} hold. If  $\beta<\zeta<\gamma<\alpha$ with
	\begin{equation}\label{eq3.14}
		K\left(L_1\left(\alpha-\gamma\right)^{-1}+L_1C_\zeta+L_2C_\zeta\right)<1,
	\end{equation}
	then there exists a mean-square random unstable invariant manifold for \eqref{eq1.1} given by \begin{equation*}
		\mathcal{M}^u\left(\tau\right)=\left\{x+h^u\left(x,\tau\right):x\in{L^2\left(\Omega,\mathcal{F}_\tau;X_{0u}\right)}\right\},
	\end{equation*} 
	where $h^u\left(\cdot,\tau\right):L^2\left(\Omega,\mathcal{F}_\tau;X_{0u}\right)\to{L^2\left(\Omega,\mathcal{F}_\tau;X_{0s}\right)}$ is a Lipschitz continuous map.
\end{theorem}
\begin{proof}We proceed it in three steps.
	
	\noindent {\bf Step 1.}
	For every $\tau\in\mathbb{R}$ and $x\in{L^2\left(\Omega,\mathcal{F}_\tau;X_{0u}\right)}$, we prove that \eqref{eq3.13} has a unique solution in ${\mathcal{C}^-_\tau}$. Denote $\mathcal{J}\left(\xi,x,\tau\right)\left(t\right)$ the right side of \eqref{eq3.13}, i.e.
	\begin{equation}\label{eq3.15}\begin{split}
			\mathcal{J}\left(\xi,x,\tau\right)\left(t\right)=&\mathbb{E}\left[T_{A_{0u}}\left(t-\tau\right)x-\int_t^\tau{T_{A_{0u}}\left(t-r\right)}\Pi_u{F}\left(\xi\left(r\right)\right)dr|\mathcal{F}_t\right]\\&+\lim_{\lambda\to+\infty}\int_{-\infty}^tT_{A_{0s}}\left(t-r\right)\lambda{R_\lambda\left(A_s\right)}\Pi_sF\left(\xi\left(r\right)\right)dr\\&+\lim_{\lambda\to+\infty}\int_{-\infty}^tT_{A_{0s}}\left(t-r\right)\lambda{R_\lambda\left(A_s\right)}\Pi_s\sigma\left(\xi\left(r\right)\right)dW\left(r\right).
		\end{split}
	\end{equation}
	We could find that $\xi$ is a solution to \eqref{eq3.13} in ${\mathcal{C}^-_\tau}$ if and only if $\xi$ is a fixed point of $\mathcal{J}\left(\xi,x,\tau\right)\left(t\right)$. It needs to show that $\mathcal{J}\left(\cdot,x,\tau\right)\left(t\right):{\mathcal{C}^-_\tau}\to{\mathcal{C}^-_\tau}$ is well defined. First we prove the process $\mathcal{J}\left(\cdot,x,\tau\right)\left(t\right),~t\leq\tau$ is predictable and continuous in $t$ in $L^2\left(\Omega,X_0\right)$. Denote $L^2_\mathcal{F}\left((-\infty,\tau],X_0\right)$ the set of all $X_0$-valued $\mathcal{F}_t$-progressively measurable processes. By \cite[Lemma 2.1]{YS1991}, there exists a unique pair $\left(y_1,y_2\right)\in{L^2_\mathcal{F}\left((-\infty,\tau],L^2\left(\Omega,X_0\right)\right)\times{L^2_\mathcal{F}\left((-\infty,\tau],\mathcal{L}_2\left(Y_0,X\right)\right)}}$ such that for all $t\leq\tau$,
	\begin{equation}
	\begin{split}\label{eq3.16}
	&	y_1\left(t\right)+\int_t^\tau{T_{A_{0u}}\left(t-r\right)}\Pi_u{F}\left(\xi\left(r\right)\right)dr+\int_t^\tau{T\left(t-r\right)}y_2\left(r\right)dW\left(r\right)\\
		&~~	=T_{A_{0u}}\left(t-\tau\right)x.
	\end{split}
	\end{equation}
	From \eqref{eq3.16}, we find that $y_1\left(t\right)$ is continuous in $t$ in $L^2\left(\Omega,X_0\right)$. Take the conditional expectation of \eqref{eq3.16}, we get that for $t\leq\tau$,
	\begin{equation}\label{eq3.17}
		y_1\left(t\right)=\mathbb{E}\left[T_{A_{0u}}\left(t-\tau\right)x-\int_t^\tau{T_{A_{0u}}\left(t-r\right)}\Pi_u{F}\left(\xi\left(r\right)\right)dr|\mathcal{F}_t\right].
	\end{equation}
	From \eqref{eq3.17}, the first term of the right side of \eqref{eq3.15} is a $X_0$-valued $\mathcal{F}_t$-progressively measurable process and continuous in $t$ in $L^2\left(\Omega,X_0\right)$. The last two terms are both $\mathcal{F}_t$-adapted and continuous in $t$ in $L^2\left(\Omega,X_0\right)$, so they are both $\mathcal{F}_t$-progressively measurable. Then we get that $\mathcal{J}\left(\xi,x,\tau\right)\left(t\right),~t\leq\tau$, is  $\mathcal{F}_t$-progressively measurable and continuous in $t$ in $L^2\left(\Omega,X_0\right)$. Secondly, we prove that $\mathcal{J}\left(\xi,x,\tau\right)\left(t\right)$ satisfies  \eqref{eq100}. We have, for $t\leq\tau$,
	\begin{equation}\begin{split}\label{eq3.18}
			& {\left\Vert{\mathcal{J}\left(\xi,x,\tau\right)\left(t\right)}\right\Vert}_{L^2\left(\Omega,X_0\right)}\\&~~\leq{\left\Vert{T_{A_{0u}}\left(t-\tau\right)x}\right\Vert}_{L^2\left(\Omega,X_0\right)}+\int_t^\tau{\left\Vert{T_{A_{0u}}\left(t-r\right)}\Pi_u{F}\left(\xi\left(r\right)\right)\right\Vert}_{L^2\left(\Omega,X_0\right)}dr\\&\quad+\lim_{\lambda\to+\infty}\int_{-\infty}^t{\left\Vert{T_{A_{0s}}\left(t-r\right)}\lambda{R_\lambda\left(A_s\right)}\Pi_s{F}\left(\xi\left(r\right)\right)\right\Vert}_{L^2\left(\Omega,X_0\right)}dr\\&\quad+{\left\Vert\lim_{\lambda\to+\infty}\int_{-\infty}^tT_{A_{0s}}\left(t-r\right)\lambda{R_\lambda\left(A_s\right)}\Pi_s\sigma\left(\xi\left(r\right)\right)dW\left(r\right)\right\Vert}_{L^2\left(\Omega,X_0\right)}.
		\end{split}
	\end{equation}
	For the first term on the right side of \eqref{eq3.18}, by \eqref{eq2.9}, we have, for $t\leq\tau$,
	\begin{equation}\label{eq3.19}
		{\left\Vert{T_{A_{0u}}\left(t-\tau\right)x}\right\Vert}_{L^2\left(\Omega,X_0\right)}\leq{Ke^{\alpha\left(t-\tau\right)}}{\left\Vert{x}\right\Vert}_{L^2\left(\Omega,X_0\right)}\leq{Ke^{\gamma\left(t-\tau\right)}}{\left\Vert{x}\right\Vert}_{L^2\left(\Omega,X_0\right)}.
	\end{equation}
	For the second term on the right side of \eqref{eq3.18}, by \eqref{eq2.9} and \eqref{eq_1}, we have, $t\leq\tau$,
	\begin{equation}\label{eq3.20}
		\begin{split}
			\int_t^\tau{\left\Vert{T_{A_{0u}}\left(t-r\right)}\Pi_u{F}\left(\xi\left(r\right)\right)\right\Vert}_{L^2\left(\Omega,X_0\right)}dr&\leq{KL_1}\int_t^\tau{e^{\alpha\left(t-r\right)}}{\left\Vert{\xi\left(r\right)}\right\Vert}_{L^2\left(\Omega,X_0\right)}dr\\&\leq{KL_1}{\left\Vert{\xi}\right\Vert}_{\mathcal{C}^{-}_\tau}\int_t^\tau{e^{\alpha{t}}}e^{\left(\gamma-\alpha\right)r}dr\\&\leq{KL_1}\left(\alpha-\gamma\right)^{-1}e^{\gamma{t}}{\left\Vert{\xi}\right\Vert}_{\mathcal{C}^{-}_\tau}.
		\end{split}
	\end{equation}
	Then it follows from \eqref{eq3.3}, \eqref{eq3.4}, \eqref{eq3.18}-\eqref{eq3.20} that  
	\begin{equation*}\begin{split}
			e^{-\gamma{t}}{\left\Vert{\mathcal{J}\left(\xi,x,\tau\right)\left(t\right)}\right\Vert}_{L^2\left(\Omega,X_0\right)}&\leq{K}e^{-\gamma\tau}{\left\Vert{x}\right\Vert}_{L^2\left(\Omega,X_0\right)}\\&+K{\left\Vert{\xi}\right\Vert}_{\mathcal{C}^{-}_\tau}\left(L_1\left(\alpha-\gamma\right)^{-1}+L_1C_\zeta+L_2C_\zeta\right).
		\end{split}
	\end{equation*}
	So we proved that $\mathcal{J}\left(\xi,x,\tau\right)\left(t\right)$ satisfies \eqref{eq100}, which means $\mathcal{J}\left(\xi,x,\tau\right)\left(t\right)\in{\mathcal{C}^-_\tau}$ and it further indicates that $\mathcal{J}\left(\cdot,x,\tau\right)\left(t\right):{\mathcal{C}^-_\tau}\to{\mathcal{C}^-_\tau}$ is well defined. Then we prove $\mathcal{J}\left(\cdot,x,\tau\right)\left(t\right):{\mathcal{C}^-_\tau}\to{\mathcal{C}^-_\tau}$ is a contraction. Given $\xi_1,\xi_2\in{\mathcal{C}^-_\tau}$, we have, for all $t\leq\tau$,
	\begin{equation}
		\label{eq3.21}
		\begin{split}
			&\left\Vert{\mathcal{J}\left(\xi_1,x,\tau\right)\left(t\right)-\mathcal{J}\left(\xi_2,x,\tau\right)\left(t\right)}\right\Vert_{L^2\left(\Omega,X_0\right)}\\&~~\leq\int_t^\tau{\left\Vert{T_{A_{0u}}\left(t-r\right)}\Pi_u\left(F\left(\xi_1\left(r\right)\right)-F\left(\xi_2\left(r\right)\right)\right)\right\Vert}_{L^2\left(\Omega,X_0\right)}dr\\&\quad+\lim_{\lambda\to+\infty}\int_{-\infty}^t{\left\Vert{T_{A_{0s}}\left(t-r\right)}\lambda{R_\lambda\left(A_s\right)}\Pi_s\left(F\left(\xi_1\left(r\right)\right)-F\left(\xi_2\left(r\right)\right)\right)\right\Vert}_{L^2\left(\Omega,X_0\right)}dr\\&\quad+{\left\Vert\lim_{\lambda\to+\infty}\int_{-\infty}^tT_{A_{0s}}\left(t-r\right)\lambda{R_\lambda\left(A_s\right)}\Pi_s\left(\sigma\left(\xi_1\left(r\right)\right)-\sigma\left(\xi_2\left(r\right)\right)\right))dW\left(r\right)\right\Vert}_{L^2\left(\Omega,X_0\right)}\\&~~=\mathcal{J}_1+\mathcal{J}_2+\mathcal{J}_3.
		\end{split}
	\end{equation}
	By \eqref{eq1} and \eqref{eq2.9}, we have, for $t\leq\tau$,
	\begin{equation}\begin{split}\label{eq3.22}
			\mathcal{J}_1&\leq{KL_1}\int_t^\tau{e^{\alpha\left(t-r\right)}}{\left\Vert{\xi_1\left(r\right)-\xi_2\left(r\right)}\right\Vert}_{L^2\left(\Omega,X_0\right)}dr\\&\leq{KL_1}{\left\Vert{\xi_1-\xi_2}\right\Vert}_{\mathcal{C}^{-}_\tau}\int_t^\tau{e^{\alpha{t}}}e^{\left(\gamma-\alpha\right)r}dr\\&\leq{KL_1}\left(\alpha-\gamma\right)^{-1}e^{\gamma{t}}{\left\Vert{\xi_1-\xi_2}\right\Vert}_{\mathcal{C}^{-}_\tau}.
		\end{split}    
	\end{equation}
	By Lemma \ref{lem0}, \eqref{eq1} and \eqref{eq2.10}, we have, for $t\leq\tau$ and $\zeta\in\left(\beta,\gamma\right)$,
	\begin{equation}\label{eq3.23}
		\begin{split}
			\mathcal{J}_2
			%   \\&={\small\lim_{\lambda\to+\infty}\lim_{g\to{-\infty}}\int_0^{t-g}{\Vert{T_{A_{0s}}\left(t-g-l\right)\lambda{\left(\lambda{I}-A_s\right)}^{-1}\Pi_s\left(F\left(\xi_1\left(l+g\right)\right)-F\left(\xi_2\left(l+g\right)\right)\right)}\Vert}_{L^2\left(\Omega,X_0\right)}{dl}}\\
			&\leq\lim_{g\to{-\infty}}L_1C_\zeta\underset{l\in[0,t-g]}{\sup}e^{\zeta\left(t-g-l\right)}{\left\Vert\xi_1\left(l+g\right)-\xi_2\left(l+g\right)\right\Vert}_{L^2\left(\Omega,X_0\right)}\\&=L_1C_\zeta\underset{q\in(-\infty,t]}{\sup}e^{\zeta\left(t-q	\right)}{\left\Vert\xi_1\left(q\right)-\xi_2\left(q\right)\right\Vert}_{L^2\left(\Omega,X_0\right)}\\&\leq{L_1}C_\zeta\underset{q\in(-\infty,t]}{\sup}e^{\zeta\left(t-q\right)}e^{\gamma{q}}{\left\Vert\xi_1\left(q\right)-\xi_2\left(q\right)\right\Vert}_{\mathcal{C}^-_\tau}\\&={L_1}C_\zeta\underset{q\in(-\infty,t]}{\sup}e^{\left(\gamma-\zeta\right)q}e^{\left(\zeta-\gamma\right)t}e^{\gamma{t}}{\left\Vert\xi_1\left(q\right)-\xi_2\left(q\right)\right\Vert}_{\mathcal{C}^-_\tau}\\&={L_1}C_\zeta{e^{\gamma{t}}}{\left\Vert\xi_1-\xi_2\right\Vert}_{\mathcal{C}^-_\tau}.
		\end{split}
	\end{equation}
	By Lemma \ref{lem0} and \eqref{eq2}, we have, for $t\leq\tau$ and $\zeta\in\left(\beta,\gamma\right)$,
	\begin{equation}\label{eq3.24}
		\begin{split}
			\mathcal{J}_3&={\left\Vert\lim_{\lambda\to+\infty}\int_{-\infty}^tT_{A_{0s}}\left(t-r\right)\lambda{R_\lambda\left(A_s\right)}\Pi_s\left(\sigma\left(\xi_1\left(r\right)\right)-\sigma\left(\xi_2\left(r\right)\right)\right)dW\left(r\right)\right\Vert}_{L^2\left(\Omega,X_0\right)}
			%   \\&=\mathbb{E}\left[\left(\lim_{\lambda\to+\infty}\lim_{g\to{-\infty}}\int_0^{t-g}T_{A_{0s}}\left(t-l-g\right)\lambda{\left(\lambda{I}-A_s\right)}^{-1}\Pi_s\left(\sigma\left(\xi_1\left(l+g\right)\right)-\sigma\left(\xi_2\left(l+g\right)\right)\right)dW\left(l\right)\right)^2\right]^{\frac{1}{2}}
			\\&\leq\lim_{g\to{-\infty}}C_\zeta\underset{l\in[0,t-g]}{\sup}e^{\zeta\left(t-g-l\right)}\mathbb{E}\left[{\left\Vert\sigma\left(\xi_1\left(l+g\right)\right)-\sigma\left(\xi_2\left(l+g\right)\right)\right\Vert}^2_{\mathcal{L}_2\left(Y_0,X\right)}\right]^{\frac{1}{2}}\\&\leq{L_2}C_\zeta\underset{q\in(-\infty,t]}{\sup}e^{\zeta\left(t-q	\right)}e^{\gamma{q}}{\left\Vert\xi_1\left(q\right)-\xi_2\left(q\right)\right\Vert}_{\mathcal{C}^-_\tau}\\&={L_2}C_\zeta\underset{q\in(-\infty,t]}{\sup}e^{\left(\gamma-\zeta\right)q}e^{\left(\zeta-\gamma\right)t}e^{\gamma{t}}{\left\Vert\xi_1\left(q\right)-\xi_2\left(q\right)\right\Vert}_{\mathcal{C}^-_\tau}\\&={L_2}C_\zeta{e^{\gamma{t}}}{\left\Vert\xi_1-\xi_2\right\Vert}_{\mathcal{C}^-_\tau}. 
		\end{split}
	\end{equation}
	By combining \eqref{eq3.22}, \eqref{eq3.23} and \eqref{eq3.24}, we have, for $t\leq\tau$, 
	\begin{equation*}
		\begin{split}
			&\left\Vert{\mathcal{J}\left(\xi_1,x,\tau\right)\left(t\right)-\mathcal{J}\left(\xi_2,x,\tau\right)\left(t\right)}\right\Vert_{L^2\left(\Omega,X_0\right)}\\&\leq{K\left(L_1\left(\alpha-\gamma\right)^{-1}+L_1C_\zeta+L_2C_\zeta\right)}e^{\gamma{t}}{\left\Vert\xi_1-\xi_2\right\Vert}_{\mathcal{C}^-_\tau}.
		\end{split} 
	\end{equation*}
	Then we have, for $t\leq\tau$,
	\begin{equation}\label{eq3.25}
		\begin{split}
			&\left\Vert{\mathcal{J}\left(\xi_1,x,\tau\right)\left(t\right)-\mathcal{J}\left(\xi_2,x,\tau\right)\left(t\right)}\right\Vert_{\mathcal{C}^-_\tau}\\&\leq{K\left(L_1\left(\alpha-\gamma\right)^{-1}+L_1C_\zeta+L_2C_\zeta\right)}{\left\Vert\xi_1-\xi_2\right\Vert}_{\mathcal{C}^-_\tau}.
		\end{split} 
	\end{equation}
	By the conditions of the theorem, $\mathcal{J}\left(\cdot,x,\tau\right)\left(t\right):{\mathcal{C}^-_\tau}\to{\mathcal{C}^-_\tau}$ is a contraction. From the contraction mapping principle, for any given $\tau\in\mathbb{R}$ and  $x\in{L^2\left(\Omega,\mathcal{F}_\tau;X_{0u}\right)}$, $\mathcal{J}\left(\cdot,x,\tau\right)\left(t\right)$ has a unique fixed point $\overline{\xi}\left(x\right)$ in $\mathcal{C}^-_\tau$. That is, for every $t\leq\tau$, $\overline\xi\left(x\right)$ is the unique solution of \eqref{eq3.15} and
	\begin{equation*}
		\mathcal{J}\left(\overline\xi, x,\tau\right)=\overline\xi\left(x\right).
	\end{equation*}
	Let $K\left(L_1\left(\alpha-\gamma\right)^{-1}+L_1C_\zeta+L_2C_\zeta\right)=\eta$. By \eqref{eq3.25}, for $x_1,x_2\in{L^2\left(\Omega,\mathcal{F}_\tau;X_{0u}\right)}$, we have, for $t\leq\tau$,
	\begin{equation}\label{eq3.26}
		\begin{split}
			&\left\Vert{\mathcal{J}\left(\overline\xi,x_1,\tau\right)\left(t\right)-\mathcal{J}\left(\overline\xi,x_2,\tau\right)\left(t\right)}\right\Vert_{\mathcal{C}^-_\tau}\\&~~=\left\Vert{\overline\xi\left(x_1\right)-\overline\xi\left(x_2\right)}\right\Vert_{\mathcal{C}^-_\tau}\\&~~\leq\underset{t\leq\tau}{\sup}~e^{-\gamma{t}}\left\Vert{T_{A_{0u}}\left(t-\tau\right)\left(x_1-x_2\right)}\right\Vert_{L^2\left(\Omega,X_0\right)}+\eta{\left\Vert\overline\xi\left(x_1\right)-\overline\xi\left(x_2\right)\right\Vert}_{\mathcal{C}^-_\tau}\\&~~\leq{K}e^{-\alpha\tau}\underset{t\leq\tau}{\sup}~e^{\left(\alpha-\gamma\right){t}}\left\Vert{x_1-x_2}\right\Vert_{L^2\left(\Omega,X_0\right)}+\eta{\left\Vert\overline\xi\left(x_1\right)-\overline\xi\left(x_2\right)\right\Vert}_{\mathcal{C}^-_\tau}\\&~~\leq{Ke^{-\gamma\tau}}\left\Vert{x_1-x_2}\right\Vert_{L^2\left(\Omega,X_0\right)}+\eta{\left\Vert\overline\xi\left(x_1\right)-\overline\xi\left(x_2\right)\right\Vert}_{\mathcal{C}^-_\tau}.
		\end{split}
	\end{equation}
	Thus by \eqref{eq3.26}, for all $x_1,x_2\in{L^2\left(\Omega,\mathcal{F}_\tau;X_{0u}\right)}$, we have
	\begin{equation}\label{eq3.27}
		{\left\Vert{\overline\xi\left(x_1\right)\left(\tau\right)-\overline\xi\left(x_2\right)\left(\tau\right)}\right\Vert}_{L^2\left(\Omega,X_0\right)}\leq{K}\left(1-\eta\right)^{-1}\left\Vert{x_1-x_2}\right\Vert_{L^2\left(\Omega,X_0\right)}.  
	\end{equation}
	\noindent {\bf Step 2.} We prove that the mean-square unstable invariant manifold is given by a graph of a Lipschitz continuous map.
	Let $h^u\left(x,\tau\right)=\Pi_{0s}\overline\xi\left(x\right)\left(\tau\right)$. Let $s\to{-\infty}$ and $t=\tau$ in \eqref{eq3.8}, we have
	\begin{equation}\label{eq3.28}
		\begin{split}
			h^u\left(x,\tau\right)=&\lim_{\lambda\to+\infty}\int_{-\infty}^\tau{T_{A_{0s}}}\left(\tau-r\right)\lambda{R_\lambda\left(A_s\right)}\Pi_s{F\left(\overline\xi\left(x\right)\left(r\right)\right)}dr\\&+\lim_{\lambda\to+\infty}\int_{-\infty}^\tau{T_{A_{0s}}}\left(\tau-r\right)\lambda{R_\lambda\left(A_s\right)}\Pi_s{\sigma\left(\overline\xi\left(x\right)\left(r\right)\right)}dW\left(r\right).
		\end{split}
	\end{equation}
	Indeed, by \eqref{eq3.23}, \eqref{eq3.24} and \eqref{eq3.27}, it follows that for all $x_1,x_2\in{L^2\left(\Omega,\mathcal{F}_\tau;X_{0u}\right)}$, 
	\begin{equation}\label{eq3.29}
		{\left\Vert{h^u\left(x_1,\tau\right)-h^u\left(x_2,\tau\right)}\right\Vert}_{L^2\left(\Omega,X_0\right)}\leq{KC_\zeta\left(L_1+L_2\right)}{\left(1-\eta\right)^{-1}}\left\Vert{x_1-x_2}\right\Vert_{L^2\left(\Omega,X_0\right)}.
	\end{equation}
	By the definition of $h^u\left(\cdot,\tau\right)$ and \eqref{eq3.29}, we get that $h^u\left(\cdot,\tau\right):L^2\left(\Omega,\mathcal{F}_\tau;X_{0u}\right)\to{L^2\left(\Omega,\mathcal{F}_\tau;X_{0s}\right)}$ is a Lipschitz continuous map. 
	Then by \eqref{eq3.15} and \eqref{eq3.28}, we have for all $\tau\in\mathbb{R}$ and $x\in{L^2\left(\Omega,\mathcal{F}_\tau;X_{0u}\right)}$,
	\begin{equation}\label{eq3.30}
		\overline\xi\left(x\right)\left(\tau\right)=x+h^u\left(x,\tau\right).
	\end{equation}
	By the definition of $\mathcal{M}^u\left(\tau\right)$ and Lemma \ref{lem3.1}, $\xi\left(\tau\right)\in\mathcal{M}^u\left(\tau\right)$ if and only if there exists $x\in
	{L^2\left(\Omega,\mathcal{F}_\tau;X_{0u}\right)}$ such that $\xi\left(\tau\right)=x+h^u\left(x,\tau\right)$. Therefore, we have
	\begin{equation*}
		\mathcal{M}^u\left(\tau\right)=\left\{x+h^u\left(x,\tau\right):x\in{L^2\left(\Omega,\mathcal{F}_\tau;X_{0u}\right)}\right\}.
	\end{equation*}
	\noindent {\bf Step 3.} We prove that $\mathcal{M}^u\left(\tau\right)$ is invariant under $\Phi$. By Definition \ref{inv}, we need to show that for each $\xi\left(\tau\right)\in\mathcal{M}^u\left(\tau\right)$ and $t_0>0$, $\Phi\left(t_0,\tau\right)\xi\left(\tau\right)\in\mathcal{M}^u\left(\tau+t_0\right)$. By the definition of $\mathcal{M}^u\left(\tau\right)$, for $\xi\left(\tau\right)\in\mathcal{M}^u\left(\tau\right)$, $\xi\in\mathcal{C}^-_\tau$ and satisfies \eqref{eq3.13} on $(-\infty,\tau]$. Let
	\begin{equation}\label{eq3.31}
		\tilde\xi\left(t\right)=\left\{
		\begin{array}{l}
			u\left(t,\tau,\xi\left(\tau\right)\right), ~\tau\leq{t}\leq\tau+t_0, \\
			\xi\left(t\right),~t<\tau. 	\\
		\end{array}\right.
	\end{equation}
	We show that $\tilde\xi\left(\tau+t_0\right)\in\mathcal{M}^u\left(\tau+t_0\right)$. Since $\xi\in\mathcal{C}^-_\tau$, it is easy to get that $\tilde\xi\in\mathcal{C}^-_{\tau+t_0}$. Then we prove that $\xi$ satisfies \eqref{eq3.13} on $(-\infty,\tau+t_0]$, which means, for all $t\leq{t_0}+\tau$, 
	\begin{equation}\label{eq3.32}
		\begin{split}
			\tilde\xi\left(t\right)=&\mathbb{E}\left[T_{A_{0u}}\left(t-\tau-t_0\right)x-\int_t^{\tau+t_0}{T_{A_{0u}}\left(t-r\right)}\Pi_u{F}\left(\tilde\xi\left(r\right)\right)d\left(r\right)|\mathcal{F}_t\right]\\&+\lim_{\lambda\to+\infty}\int_{-\infty}^tT_{A_{0s}}\left(t-r\right)\lambda{R_\lambda\left(A_s\right)}\Pi_sF\left(\tilde\xi\left(r\right)\right)dr\\&+\lim_{\lambda\to+\infty}\int_{-\infty}^tT_{A_{0s}}\left(t-r\right)\lambda{R_\lambda\left(A_s\right)}\Pi_s\sigma\left(\tilde\xi\left(r\right)\right)dW\left(r\right)  
		\end{split}
	\end{equation}
	in $L^2\left(\Omega,X_0\right)$ with $x=\Pi_{0u}\tilde\xi\left(\tau+t_0\right)\in{L^2\left(\Omega,\mathcal{F}_{\tau+t_0};X_{0u}\right)}$. By \eqref{eq3.31} and \eqref{eq2.5}, we have
	\begin{equation}\begin{split}\label{eq3.33}
			\tilde\xi\left(\tau+t_0\right)=&T\left(t_0\right)\xi\left(\tau\right)+\lim_{\lambda\to+\infty}\int_\tau^{\tau+t_0}T\left(\tau+t_0-r\right)\lambda{R_\lambda\left(A\right)}F\left(\tilde\xi\left(r\right)\right)dr\\&+\lim_{\lambda\to+\infty}\int_\tau^{\tau+t_0}T\left(\tau+t_0-r\right)\lambda{R_\lambda\left(A\right)}\sigma\left(\tilde\xi\left(r\right)\right)dW\left(r\right).
		\end{split}
	\end{equation}
	Then we use the same method of getting \eqref{eq3.7} to apply $\Pi_{0u}$ to \eqref{eq3.33}, we have
	\begin{equation}\label{eq3.34}
		\begin{split}
			x=&\Pi_{0u}\tilde\xi\left(\tau+t_0\right)=T\left(t_0\right)\Pi_{0u}\xi\left(\tau\right)+\int_\tau^{\tau+t_0}T\left(\tau+t_0-r\right)\Pi_u{F}\left(\tilde\xi\left(r\right)\right)dr\\&+\int_\tau^{\tau+t_0}T\left(\tau+t_0-r\right)\Pi_u\sigma\left(\tilde\xi\left(r\right)\right)dW\left(r\right).   
		\end{split}
	\end{equation}
	Then \eqref{eq3.32} can be rewritten as, for $t\leq\tau+t_0$,
	\begin{equation}\label{eq3.35}
		\begin{split}
			\tilde\xi\left(t\right)=&\mathbb{E}\left[T_{A_{0u}}\left(t-\tau\right)\Pi_{0u}\xi\left(\tau\right)-\int_t^\tau{T_{A_{0u}}\left(t-r\right)}\Pi_u{F}\left(\tilde\xi\left(r\right)\right)d\left(r\right)|\mathcal{F}_t\right]\\&-\mathbb{E}\left[\int_t^\tau{T_{A_{0u}}\left(t-r\right)}\Pi_u\sigma\left(\tilde\xi\left(r\right)\right)dW\left(r\right)|\mathcal{F}_t\right]\\&+\lim_{\lambda\to+\infty}\int_{-\infty}^tT_{A_{0s}}\left(t-r\right)\lambda{R_\lambda\left(A_s\right)}\Pi_sF\left(\tilde\xi\left(r\right)\right)dr\\&+\lim_{\lambda\to+\infty}\int_{-\infty}^tT_{A_{0s}}\left(t-r\right)\lambda{R_\lambda\left(A_s\right)}\Pi_s\sigma\left(\tilde\xi\left(r\right)\right)dW\left(r\right).
		\end{split}
	\end{equation}
	Since $\xi\in\mathcal{C}^-_\tau$ satisfies \eqref{eq3.13} on $(-\infty,\tau]$, then $\tilde\xi$ satisfies \eqref{eq3.35} for $t\leq\tau$. It remains to show that $\tilde\xi$ satisfies \eqref{eq3.35} for $t>\tau$. By \eqref{eq3.31} and \eqref{eq2.5}, for $t>\tau$,
	\begin{equation}\label{eq3.36}
		\begin{split}
			\tilde\xi\left(t\right)=&T\left(t-\tau\right)\xi\left(\tau\right)+\lim_{\lambda\to+\infty}\int_\tau^tT\left(t-r\right)\lambda{R_\lambda\left(A\right)}F\left(\tilde\xi\left(r\right)\right)dr\\&+\lim_{\lambda\to+\infty}\int_\tau^tT\left(t-r\right)\lambda{R_\lambda\left(A\right)}\sigma\left(\tilde\xi\left(r\right)\right)dW\left(r\right).    
		\end{split}
	\end{equation}
	Since $\xi\in\mathcal{C}^-_\tau$ satisfies \eqref{eq3.13} with $x=\Pi_{0u}\xi\left(\tau\right)$ on $(-\infty,\tau]$, let $t=\tau$ in \eqref{eq3.13}, we have,
	\begin{equation}\label{eq3.37}
		\begin{split}
			\xi\left(\tau\right)=&\Pi_{0u}\xi\left(\tau\right)+\lim_{\lambda\to+\infty}\int_{-\infty}^\tau{T_{A_{0s}}}\left(\tau-r\right)\lambda{R_\lambda\left(A_s\right)}\Pi_sF\left(\xi\left(r\right)\right)dr\\&+\lim_{\lambda\to+\infty}\int_{-\infty}^\tau{T_{A_{0s}}}\left(\tau-r\right)\lambda{R_\lambda\left(A_s\right)}\Pi_s\sigma\left(\xi\left(r\right)\right)dW\left(r\right)\\=&\Pi_{0u}\xi\left(\tau\right)+\lim_{\lambda\to+\infty}\int_{-\infty}^\tau{T_{A_{0s}}}\left(\tau-r\right)\lambda{R_\lambda\left(A_s\right)}\Pi_sF\left(\tilde\xi\left(r\right)\right)dr\\&+\lim_{\lambda\to+\infty}\int_{-\infty}^\tau{T_{A_{0s}}}\left(\tau-r\right)\lambda{R_\lambda\left(A_s\right)}\Pi_s\sigma\left(\tilde\xi\left(r\right)\right)dW\left(r\right).
		\end{split}
	\end{equation}
	By \eqref{eq3.36}, \eqref{eq3.37} and Lemma \ref{lem2.3}(i), we have, for $t>\tau$,
	\begin{equation}\label{eq3.38}
		\begin{split}
			\tilde\xi\left(t\right)=&T\left(t-\tau\right)\Pi_{0u}\xi\left(\tau\right)+\lim_{\lambda\to+\infty}\int_{-\infty}^\tau{T_{A_{0s}}}\left(t-r\right)\lambda{R_\lambda\left(A_s\right)}\Pi_sF\left(\tilde\xi\left(r\right)\right)dr\\&+\lim_{\lambda\to+\infty}\int_{-\infty}^\tau{T_{A_{0s}}}\left(t-r\right)\lambda{R_\lambda\left(A_s\right)}\Pi_s\sigma\left(\tilde\xi\left(r\right)\right)dW\left(r\right)\\&+\lim_{\lambda\to+\infty}\int_\tau^tT\left(t-r\right)\lambda{R_\lambda\left(A\right)}F\left(\tilde\xi\left(r\right)\right)dr\\&+\lim_{\lambda\to+\infty}\int_\tau^tT\left(t-r\right)\lambda{R_\lambda\left(A\right)}\sigma\left(\tilde\xi\left(r\right)\right)dW\left(r\right)\\=&T\left(t-\tau\right)\Pi_{0u}\xi\left(\tau\right)+\lim_{\lambda\to+\infty}\int_{-\infty}^t{T_{A_{0s}}}\left(t-r\right)\lambda{R_\lambda\left(A_s\right)}\Pi_sF\left(\tilde\xi\left(r\right)\right)dr\\&+\lim_{\lambda\to+\infty}\int_{-\infty}^t{T_{A_{0s}}}\left(t-r\right)\lambda{R_\lambda\left(A_s\right)}\Pi_s\sigma\left(\tilde\xi\left(r\right)\right)dW\left(r\right)\\&+\int_\tau^tT_{A_{0u}}\left(t-r\right)\Pi_uF\left(\tilde\xi\left(r\right)\right)dr+\int_\tau^tT_{A_{0u}}\left(t-r\right)\Pi_u\sigma\left(\tilde\xi\left(r\right)\right)dW\left(r\right). 
		\end{split}
	\end{equation}
	Since $\tilde\xi\left(t\right)=u\left(t,\tau,\xi\left(\tau\right)\right)$ is $\mathcal{F}_t$-adapted for $t>\tau$, all terms of \eqref{eq3.38} are $\mathcal{F}_t$-measurable for $t>\tau$, so for $t>\tau$, take $\mathbb{E}\left(\cdot|\mathcal{F}_t\right)$ on \eqref{eq3.38}, we get \eqref{eq3.35}. So far we proved $\tilde\xi$ satisfies \eqref{eq3.35}, which means it satisfies \eqref{eq3.32} and \eqref{eq3.13} as well for all $t\leq\tau+t_0$. So we get $\tilde\xi\in\mathcal{C}^-_{\tau+t_0}$ satisfies \eqref{eq3.13} on $(-\infty,\tau+t_0]$, by the definition of $\mathcal{M}^u\left(\tau\right)$, we get $\tilde\xi\left(\tau+t_0\right)\in\mathcal{M}^u\left(\tau+t_0\right)$. By \eqref{eq3.31}, 
	\begin{equation*}
		\Phi\left(t_0,\tau\right)\left(\xi\left(\tau\right)\right)=u\left(\tau+t_0,\tau,\xi\left(\tau\right)\right)=\tilde\xi\left(\tau+t_0\right).
	\end{equation*}
	So $\Phi\left(t_0,\tau\right)\left(\xi\left(\tau\right)\right)\in\mathcal{M}^u\left(\tau+t_0\right)$, implying $\Phi\left(t_0,\tau\right)\mathcal{M}^u\left(\tau\right)\subseteq{\mathcal{M}^u\left(\tau+t_0\right)}$ for all $t_0>0$. So $\mathcal{M}^u\left(\tau\right)$ is invariant under $\Phi$.
\end{proof}
\section{Mean-square random stable invariant sets}\label{sec4}
This section is dedicated to the existence of mean-square random stable sets of \eqref{eq1.1} through the same method. Given $\tau\in\mathbb{R}$, denote by $\mathcal{C}^+_{\tau}$ the space of all $X_0$-valued $\mathcal{F}_t$-progressively measurable  processes $\xi\left(t\right):t\in[\tau,+\infty)$ such that $\xi:[\tau,+\infty)\to{L^2\left(\Omega,X_0\right)}$ is continuous and \begin{equation}\label{eq4.1}
	\underset{t\geq\tau}{\sup}\left(e^{-\gamma{t}}{\left\Vert\xi\left(t\right)\right\Vert}_{L^2\left(\Omega,X_0\right)}\right)<\infty,
\end{equation}
with norm
\begin{equation*}
	{\left\Vert\xi\left(t\right)\right\Vert}_{\mathcal{C}^+_\tau}=\underset{t\geq\tau}{\sup}\left(e^{-\gamma{t}}{\left\Vert\xi\left(t\right)\right\Vert}_{L^2\left(\Omega,X_0\right)}\right),~\forall{\xi\in{\mathcal{C}^+_{\tau}}}.
\end{equation*}
where $\gamma\in\left(\beta,\alpha\right)$. Notice that ${\mathcal{C}^+_\tau}$ is a Banach space.
To construct mean-square random stable invariant sets of \eqref{eq1.1}, we need to find all solutions of \eqref{eq1.1} in the space ${\mathcal{C}^+_\tau}$. 
\begin{lemma}\label{lem4.1}
	Let Assumption \ref{as1}, Assumption \ref{as2.1}, Assumption \ref{as2.3} and Assumption \ref{as2.5} be satisfied. $\xi\in{\mathcal{C}^+_\tau}$ for some $\tau\in\mathbb{R}$. Then $\xi$ is an integrated solution of \eqref{eq1.1} on $[\tau,+\infty)$ if and only if there exists $x\in{L^2\left(\Omega,\mathcal{F}_\tau;X_{0s}\right)}$ such that for all $t\geq\tau$,
	\begin{equation}\label{eq4.2}
		\begin{split}
			\xi\left(t\right)=&T_{A_{0s}}\left(t-\tau\right)x-\int_t^\infty{T_{A_{0u}}\left(t-r\right)}\Pi_uF\left(\xi\left(r\right)\right)dr\\&-\int_t^\infty{T_{A_{0u}}\left(t-r\right)}\Pi_u\sigma\left(\xi\left(r\right)\right)dW\left(r\right)\\&+\lim_{\lambda\to+\infty}\int_\tau^tT_{A_{0s}}\left(t-r\right)\lambda{R_\lambda\left(A_s\right)}\Pi_sF\left(\xi\left(r\right)\right)dr\\&+\lim_{\lambda\to+\infty}\int_\tau^tT_{A_{0s}}\left(t-r\right)\lambda{R_\lambda\left(A_s\right)}\Pi_s\sigma\left(\xi\left(r\right)\right)dW\left(r\right)
		\end{split}    
	\end{equation}
	in $L^2\left(\Omega,X_0\right)$.
\end{lemma}
\begin{proof}
	\noindent {\bf Step 1: necessity.} Suppose $\xi\in{\mathcal{C}^+_\tau}$ with $\tau\in\mathbb{R}$ is an integrated solution of \eqref{eq1.1}. We prove that \eqref{eq4.2} holds true. Since $\xi\in{\mathcal{C}^+_\tau}$, by Lemma \ref{lem0} and Assumption \ref{as2.1}, the last two integrals on the right side of \eqref{eq4.2} are well-defined in $L^2\left(\Omega,X_0\right)$. For the first integral, according to \eqref{eq_1} and \eqref{eq2.9},
	\begin{equation}\label{eq4.3}
		\begin{split}
			\int_t^\infty{\left\Vert{T_{A_{0u}}\left(t-r\right)}\Pi_uF\left(\xi\left(r\right)\right)\right\Vert}_{L^2\left(\Omega,X_0\right)}dr&\leq{KL_1}\int_t^\infty{e^{\alpha\left(t-r\right)}}\left\Vert\xi\left(r\right)\right\Vert_{L^2\left(\Omega,X_0\right)}dr\\&\leq{KL_1}\left\Vert\xi\left(r\right)\right\Vert_{\mathcal{C}^+_\tau}e^{\alpha{t}}\int_t^\infty{e^{-\left(\alpha-\gamma\right)r}}dr\\&\leq{KL_1}\left(\alpha-\gamma\right)^{-1}e^{\gamma{t}}\left\Vert\xi\right\Vert_{\mathcal{C}^+_\tau}\\&<\infty.
		\end{split}
	\end{equation}
	By \eqref{eq4.3}, for $t\geq\tau$,
	\begin{equation*}
		\int_t^\infty{T_{A_{0u}}\left(t-r\right)}\Pi_uF\left(\xi\left(r\right)\right)dr
	\end{equation*}
	is well-defined in $L^2\left(\Omega,X_0\right)$. For the second integral, according to \eqref{eq_2} and \eqref{eq2.9}, 
	\begin{equation}\label{eq4.4}
		\begin{split}
			&\left\Vert\int_t^\infty{T_{A_{0u}}\left(t-r\right)}\Pi_u\sigma\left(\xi\left(r\right)\right)dW\left(r\right)\right\Vert_{L^2\left(\Omega,X_0\right)}\\&~~\leq\left(\mathbb{E}\left[\int_t^\infty\left\Vert{T_{A_{0u}}\left(t-r\right)}\Pi_u\sigma\left(\xi\left(r\right)\right)\right\Vert^2_{\mathcal{L}_2\left(Y_0,X\right)}dr\right]\right)^{\frac{1}{2}}\\&~~\leq{KL_2}\left(\int_t^\infty{e^{2\alpha\left(t-r\right)}}\mathbb{E}\left[\left\Vert\xi\left(r\right)\right\Vert^2_{L^2\left(\Omega,X_0\right)}\right]dr\right)^{\frac{1}{2}}\\&~~\leq{KL_2}\left\Vert\xi\left(r\right)\right\Vert_{\mathcal{C}^+_\tau}e^{\alpha{t}}\left(\int_t^\infty{e^{-2\left(\alpha-\gamma\right)r}}dr\right)^{\frac{1}{2}}\\&~~\leq\frac{1}{\sqrt{2}}{KL_2}\left(\alpha-\gamma\right)^{-\frac{1}{2}}e^{\gamma{t}}\left\Vert\xi\right\Vert_{\mathcal{C}^+_\tau}\\&~~<\infty.
		\end{split}
	\end{equation}
	By \eqref{eq4.4}, for $t\geq\tau$,
	\begin{equation*}
		\int_t^\infty{T_{A_{0u}}\left(t-r\right)}\Pi_u\sigma\left(\xi\left(r\right)\right)dW\left(r\right)
	\end{equation*}
	is well-defined in $L^2\left(\Omega,X_0\right)$. 
	Since $\xi\in{\mathcal{C}^+_\tau}$ is a solution of \eqref{eq1.1} on $[\tau,+\infty)$, by \eqref{eq2.5}, for all $t\geq\tau$, 
	\begin{equation}\label{eq4.5}
		\begin{split}
			\xi\left(t\right)=&T\left(t-\tau\right)\xi\left(\tau\right)+\lim_{\lambda\to+\infty}\int_\tau^tT\left(t-r\right){R_\lambda\left(A\right)}F\left(\xi\left(r\right)\right)dr\\&+\lim_{\lambda\to+\infty}\int_\tau^tT\left(t-r\right)\lambda{R_\lambda\left(A\right)}\sigma\left(\xi\left(r\right)\right)dW\left(r\right).
		\end{split}
	\end{equation}
	Then we apply projections $\Pi_{0u}$ and $\Pi_{0s}$ to \eqref{eq4.5}. Note that $\Pi_u:X\to{X}$ is a bounded linear projection satisfying
	\begin{equation*}
		\Pi_u{R_\lambda\left(A\right)}={R_\lambda\left(A\right)}\Pi_u,~\forall{\lambda>\vartheta}.
	\end{equation*}
	Also, $A|_{\Pi_u\left(X\right)}=A|_{\Pi_u\left(X\right)}=A|_{X_u}=A_u$ satisfies Assumption \ref{as2.1} in $\Pi_u\left(X\right)=X_u$. Moreover, we know that $\Pi_u$ has a finite rank and $\Pi_u\left(X\right)\subset{X_0}$, which further indicates that  $A_0|_{\Pi_u\left(X_0\right)}=A_0|_{X_{0u}}=A_{0u}$ and $A|_{\Pi_u\left(X\right)}=A|_{X_u}=A_u$. Thus if for each $x\in{X_0}$ and each $t\geq\tau$, then the map $t\to\Pi_u{\xi\left(t\right)}$ solves
	\begin{equation*}
		d\Pi_u{\xi\left(t\right)}=A_{0u}\Pi_u{\xi\left(t\right)}dt+\Pi_u{F\left(\xi\left(t\right)\right)}dt+\Pi_u{\sigma\left(\xi\left(t\right)\right)}dW\left(t\right)
	\end{equation*}
	in $\Pi_u\left(X_0\right)=X_{0u}$, then for all $t\geq\tau$, 
	\begin{equation}\label{eq4.6}
		\begin{split}
			\Pi_u{\xi\left(t\right)}=&T_{A_0|_{\Pi_s\left(X_0\right)}}\left(t-\tau\right)\Pi_u{\xi\left(\tau\right)}+\left(S_{{A|}_{\Pi_u\left(X\right)}}\diamond{\Pi_u{\left(F\left(\xi\right)+\sigma\left(\xi\right)dW\right)}}\right)\left(r\right)\\=&T_{A_{0s}}\left(t-\tau\right)\Pi_u{\xi\left(\tau\right)}+\left(S_{{A}_{u}}\diamond{\Pi_u{\left(F\left(\xi\right)+\sigma\left(\xi\right)dW\right)}}\right)\left(r\right)\\=&T_{A_{0u}}\left(t-\tau\right)\Pi_u{\xi\left(\tau\right)}+\lim_{\lambda\to+\infty}\int_\tau^tT_{A_{0u}}\left(t-r\right)\lambda{R_\lambda\left(A_u\right)}\Pi_u{F\left(\xi\left(r\right)\right)}dr\\&+\lim_{\lambda\to+\infty}\int_\tau^tT_{A_{0u}}\left(t-r\right)\lambda{R_\lambda\left(A_u\right)}\Pi_u{\sigma\left(\xi\left(r\right)\right)}dW\left(r\right).
		\end{split}
	\end{equation}
	Since $\Pi_u$ is a extension of $\Pi_{0u}$ from $X_0$ to $X$ and $\Pi_u\left(X\right)\subset{X_0}$, by \eqref{eq4.6} and Lemma \ref{lem2.3}(i), we have
	\begin{equation}\label{eq4.7}
		\begin{split}
			\Pi_{0u}{\xi\left(t\right)}=&\Pi_u{\xi\left(t\right)}=T_{A_{0u}}\left(t-\tau\right)\Pi_{0u}{\xi\left(\tau\right)}+\int_\tau^tT_{A_{0u}}\left(t-r\right)\Pi_u{F\left(\xi\left(r\right)\right)}dr\\&+\int_\tau^tT_{A_{0u}}\left(t-r\right)\Pi_u{\sigma\left(\xi\left(r\right)\right)}dW\left(r\right).
		\end{split}    
	\end{equation}
	Through similar arguments, we have
	\begin{equation}\label{eq4.8}
		\begin{split}
			\Pi_{0s}{\xi\left(t\right)}=&T_{A_{0s}}\left(t-\tau\right)\Pi_{0s}{\xi\left(\tau\right)}+\lim_{\lambda\to+\infty}\int_\tau^tT_{A_{0s}}\left(t-r\right)\lambda{R_\lambda\left(A_s\right)}\Pi_s{F\left(\xi\left(r\right)\right)}dr\\&+\lim_{\lambda\to+\infty}\int_\tau^tT_{A_{0s}}\left(t-r\right)\lambda{R_\lambda\left(A_s\right)}\Pi_s{\sigma\left(\xi\left(r\right)\right)}dW\left(r\right).
		\end{split}
	\end{equation}
	Let $t=s$ in \eqref{eq4.7}, we have, for all $s\geq\tau$,
	\begin{equation}\label{eq4.9}
		\begin{split}
			\Pi_{0u}{\xi\left(s\right)}=&T_{A_{0u}}\left(s-\tau\right)\Pi_{0u}{\xi\left(\tau\right)}+\int_\tau^sT_{A_{0u}}\left(s-r\right)\Pi_u{F\left(\xi\left(r\right)\right)}dr\\&+\int_\tau^sT_{A_{0u}}\left(s-r\right)\Pi_u{\sigma\left(\xi\left(r\right)\right)}dW\left(r\right).
		\end{split}
	\end{equation}
	Hence for all $s\geq{t}\geq\tau$,
	\begin{equation}\label{eq4.10}
		\begin{split}
			T_{A_{0u}}\left(t-s\right)\Pi_{0u}{\xi\left(s\right)}=&T_{A_{0u}}\left(t-\tau\right)\Pi_{0u}{\xi\left(\tau\right)}+\int_\tau^sT_{A_{0u}}\left(t-r\right)\Pi_u{F\left(\xi\left(r\right)\right)}dr\\&+\int_\tau^sT_{A_{0u}}\left(t-r\right)\Pi_u{\sigma\left(\xi\left(r\right)\right)}dW\left(r\right).
		\end{split}
	\end{equation}
	By \eqref{eq4.7} and \eqref{eq4.10}, we have, for all $s\geq{t}\geq\tau$,
	\begin{equation}\label{eq4.11}
		\begin{split}
			\Pi_{0u}{\xi\left(t\right)}=&T_{A_{0u}}\left(t-s\right)\Pi_{0u}{\xi\left(s\right)}-\int_t^s{T_{A_{0u}}\left(t-r\right)}\Pi_uF\left(\xi\left(r\right)\right)dr\\&-\int_t^s{T_{A_{0u}}\left(t-r\right)}\Pi_u\sigma\left(\xi\left(r\right)\right)dW\left(r\right). 
		\end{split}
	\end{equation}
	By \eqref{eq4.8} and \eqref{eq4.11}, for all $s\geq{t}\geq\tau$, we have
	\begin{equation}\label{eq4.12}
		\begin{split}
			\xi\left(t\right)=&T_{A_{0s}}\left(t-\tau\right)\Pi_{0s}\xi\left(\tau\right)-\int_t^s{T_{A_{0u}}\left(t-r\right)}\Pi_uF\left(\xi\left(r\right)\right)dr\\&-\int_t^s{T_{A_{0u}}\left(t-r\right)}\Pi_u\sigma\left(\xi\left(r\right)\right)dW\left(r\right)+T_{A_{0u}}\left(t-s\right)\Pi_{0u}{\xi\left(s\right)}\\&+\lim_{\lambda\to+\infty}\int_\tau^tT_{A_{0s}}\left(t-r\right)\lambda{R_\lambda\left(A_s\right)}\Pi_s{F\left(\xi\left(r\right)\right)}dr\\&+\lim_{\lambda\to+\infty}\int_\tau^tT_{A_{0s}}\left(t-r\right)\lambda{R_\lambda\left(A_s\right)}\Pi_s{\sigma\left(\xi\left(r\right)\right)}dW\left(r\right).
		\end{split}
	\end{equation}
	Since for all $s\geq{t}\geq\tau$, let $s\to{+\infty}$ in \eqref{eq4.12}, then by \eqref{eq2.9},
	\begin{equation*}
		\begin{split}
			{\left\Vert{T_{A_{0u}}\left(t-s\right)\Pi_{0u}{\xi\left(s\right)}}\right\Vert}_{L^2\left(\Omega,X_0\right)}&\leq{Ke^{\alpha\left(t-s\right)}}{\left\Vert{\xi\left(s\right)}\right\Vert}_{L^2\left(\Omega,X_0\right)}\\&\leq{Ke^{\alpha{t}}e^{\left(\gamma-\alpha\right)s}}{\left\Vert{\xi\left(s\right)}\right\Vert}_{\mathcal{C}^-_\tau}\to0.
		\end{split}
	\end{equation*}
	Thus \eqref{eq4.2} is fulfilled as wished with $x=\Pi_{0s}\xi\left(\tau\right)$. 
	
	\noindent {\bf Step 2: sufficiency.} Suppose $\xi\in{\mathcal{C}^+_\tau}$ satisfies \eqref{eq4.2} for $t\geq\tau$. Since $\xi\in{\mathcal{C}^+_\tau}$, we know $\xi$ is an $X_0$-valued $\mathcal{F}_t$-progressively measurable process and $\xi\in{C\left([\tau,+\infty),L^2\left(\Omega,X_0\right)\right)}$. And from \eqref{eq4.2}, one can verify that for all $t\geq\tau$, \begin{equation}\label{eq4.13}
		\begin{split}
			\xi\left(t\right)=&T\left(t-\tau\right)\xi\left(\tau\right)+\lim_{\lambda\to+\infty}\int_\tau^tT\left(t-s\right)\lambda{R_\lambda\left(A\right)}F\left(\xi\left(s\right)\right)ds\\&+\lim_{\lambda\to+\infty}\int_\tau^tT\left(t-s\right)\lambda{R_\lambda\left(A\right)}\sigma\left(\xi\left(s\right)\right)dW\left(s\right).
		\end{split}
	\end{equation}
	From \eqref{eq4.13}, we get that for all $t\geq\tau$,
	\begin{equation*}
		\xi\left(t\right)=u\left(t,\tau,\xi\left(\tau\right)\right),
	\end{equation*}
	where $u$ is the integrated solution of \eqref{eq1.1} with initial condition $\xi\left(\tau\right)$ at initial time $\tau$. Then by the definition of integrated solutions of \eqref{eq1.1} on $[\tau,+\infty)$, $\xi$ is an integrated solution of \eqref{eq1.1} on $[\tau,+\infty)$.
\end{proof}
Now we are ready to construct the mean-square random stable invariant sets of \eqref{eq1.1}. By Lemma \ref{lem4.1}, if $\xi\in{\mathcal{C}^+_\tau}$ is an integrated solution of \eqref{eq1.1}, then $\xi$ must satisfy \eqref{eq4.2}. Take the conditional expectation $\mathbb{E}\left(\cdot|\mathcal{F}_t\right)$ on 
\eqref{eq4.2}, we get for $t\geq\tau$,
\begin{equation}\label{eq4.14}
	\begin{split}
		\xi\left(t\right)=&T_{A_{0s}}\left(t-\tau\right)x-\mathbb{E}\left[\int_t^\infty{T_{A_{0u}}\left(t-r\right)}\Pi_uF\left(\xi\left(r\right)\right)dr|\mathcal{F}_t\right]\\&-\mathbb{E}\left[\int_t^\infty{T_{A_{0u}}\left(t-r\right)}\Pi_u\sigma\left(\xi\left(r\right)\right)dW\left(r\right)|\mathcal{F}_t\right]\\&+\lim_{\lambda\to+\infty}\int_\tau^tT_{A_{0s}}\left(t-r\right)\lambda{R_\lambda\left(A_s\right)}\Pi_sF\left(\xi\left(r\right)\right)dr\\&+\lim_{\lambda\to+\infty}\int_\tau^tT_{A_{0s}}\left(t-r\right)\lambda{R_\lambda\left(A_s\right)}\Pi_s\sigma\left(\xi\left(r\right)\right)dW\left(r\right).
	\end{split}
\end{equation}
in $L^2\left(\Omega,X_0\right)$ with $x\in{L^2\left(\Omega,\mathcal{F}_\tau;X_{0s}\right)}$. Similar to what we did when constructing unstable manifolds, given $\tau\in\mathbb{R}$, set
\begin{equation*}
	\mathcal{M}^{s*}\left(\tau\right)=\left\{\xi\left(\tau\right)\in{{L^2\left(\Omega,\mathcal{F}_\tau;X_0\right)}}:\xi\in{\mathcal{C}^+_\tau} \text{and satisfies}~\eqref{eq4.14}~\text{on}~[\tau,+\infty)\right\}.
\end{equation*}
But it turns out that the invariance property of $\mathcal{M}^{s*}\left(\tau\right)$ cannot be proved due to the $\mathcal{F}_t$-adaptedness issues, which means we cannot give the existence of mean-square random stable invariant manifolds yet. We shall consider the mean-square random stable invariant sets instead. For every $\tau\in\mathbb{R}$ and $t\geq\tau$, set
\begin{equation}\label{eq4.15}
\begin{split}
&\mathcal{M}^{s}\left(\tau\right)\\
&~=\left\{\xi\left(\tau\right)\in{{L^2\left(\Omega,\mathcal{F}_\tau;X_0\right)}}:\xi\left(t\right)\in{\mathcal{C}^+_\tau}~\text{and is an integrated solution of}~\eqref{eq1.1}\right\}.
\end{split}
\end{equation}
We shall prove that $\mathcal{M}^s\left(\tau\right)$ is a mean-square random stable invariant set of \eqref{eq1.1}. By Lemma \ref{lem4.1}, we need to find all solutions of \eqref{eq4.2} which belongs to ${\mathcal{C}^+_\tau}$ on $[\tau,+\infty)$. Given $\tau\in\mathbb{R}$, define a subset of ${L^2\left(\Omega,\mathcal{F}_\tau;X_{0s}\right)}$ as $X_{0s}^\tau$ which is given by 
\begin{equation*}
	X_{0s}^\tau=\left\{x\in{L^2\left(\Omega,\mathcal{F}_\tau;X_{0s}\right)}:\text{there exists}~ \xi\in{\mathcal{C}^+_\tau}~\text{satisfying}~\eqref{eq4.2}\right\}.
\end{equation*}
$X_{0s}^\tau$ is nonempty since $0\in{X_{0s}^\tau}$. By \eqref{eq4.15} and Lemma \ref{lem4.1}, we get that, for every given $\tau\in\mathbb{R}$ and $t\geq\tau$,
\begin{equation}\label{eq4.16}
\begin{split}
&\mathcal{M}^{s}\left(\tau\right)\\
&~=\left\{\xi\left(\tau\right)\in{{L^2\left(\Omega,\mathcal{F}_\tau;X_0\right)}}:\exists ~x\in{X_{0s}^\tau}~\text{such that}~\xi\left(t\right)\in{\mathcal{C}^+_\tau}~\text{satisfies}~\eqref{eq4.2}\right\}.
\end{split}
\end{equation}
We below prove that $ \mathcal{M}^s\left(\tau\right)$ is a invariant set of $\Phi$ and it is given by a graph of  Lipschitz function.
\begin{theorem}\label{thm4.2}
	Assume Assumption \ref{as1}, Assumption \ref{as2.1}, Assumption \ref{as2.3} and Assumption \ref{as2.5} hold. If  $\beta<\zeta<\gamma<\alpha$ with
	\begin{equation}\label{eq4.17}
		{K\left(L_1\left(\alpha-\gamma\right)^{-1}+L_2\left(2\alpha-2\gamma\right)^{-\frac{1}{2}}+L_1C_\zeta+L_2C_\zeta\right)}<1,
	\end{equation}
	then there exists a mean-square random stable invariant set for \eqref{eq1.1} given by \begin{equation*}
		\mathcal{M}^s\left(\tau\right)=\left\{x+h^s\left(x,\tau\right):x\in{L^2\left(\Omega,\mathcal{F}_\tau;X_{0s}\right)}\right\},
	\end{equation*} 
	where $h^s\left(\cdot,\tau\right):X_{0s}^\tau\to{L^2\left(\Omega,\mathcal{F}_\tau;X_{0u}\right)}$ is a Lipschitz continuous map.
\end{theorem}
\begin{proof}We proceed it in three steps.
	
	\noindent {\bf Step 1.} By the definition of $X_{0s}^\tau$, if $x\in{X_{0s}^\tau}$, then there exists at least one $\xi\in{\mathcal{C}^+_\tau}$ which satisfies \eqref{eq4.2}. We prove that when \eqref{eq4.17} is satisfied, such $\xi$ is unique. Suppose $\xi_1,\xi_2\in{\mathcal{C}^+_\tau}$ satisfy \eqref{eq4.2} with $x\in{X_{0s}^\tau}$, then by \eqref{eq4.2}, we have, for $t\geq\tau$,
	\begin{equation}\label{eq4.18}
		\begin{split}
			&\left\Vert{\xi_1\left(t\right)-\xi_2\left(t\right)}\right\Vert_{L^2\left(\Omega,X_0\right)}\\&~~\leq\int_t^\infty{\left\Vert{T_{A_{0u}}\left(t-r\right)}\Pi_u\left(F\left(\xi_1\left(r\right)\right)-F\left(\xi_2\left(r\right)\right)\right)\right\Vert}_{L^2\left(\Omega,X_0\right)}dr\\&~~\leq{\left\Vert\int_t^\infty{T_{A_{0u}}}\left(t-r\right)\Pi_u\left(\sigma\left(\xi_1\left(r\right)\right)-\sigma\left(\xi_2\left(r\right)\right)\right))dW\left(r\right)\right\Vert}_{L^2\left(\Omega,X_0\right)}\\&\quad+\lim_{\lambda\to+\infty}\int_\tau^t{\left\Vert{T_{A_{0s}}\left(t-r\right)}\lambda{R_\lambda\left(A_s\right)}\Pi_s\left(F\left(\xi_1\left(r\right)\right)-F\left(\xi_2\left(r\right)\right)\right)\right\Vert}_{L^2\left(\Omega,X_0\right)}dr\\&\quad+{\left\Vert\lim_{\lambda\to+\infty}\int_\tau^tT_{A_{0s}}\left(t-r\right)\lambda{R_\lambda\left(A_s\right)}\Pi_s\left(\sigma\left(\xi_1\left(r\right)\right)-\sigma\left(\xi_2\left(r\right)\right)\right))dW\left(r\right)\right\Vert}_{L^2\left(\Omega,X_0\right)}\\&~~=\mathcal{I}_1+\mathcal{I}_2+\mathcal{I}_3+\mathcal{I}_4.
		\end{split}
	\end{equation}
	By \eqref{eq1} and \eqref{eq2.9}, we have, for $t\geq\tau$,
	\begin{equation}\label{eq4.19}
		\begin{split}
			\mathcal{I}_1&\leq{KL_1}\int_t^\infty{e^{\alpha\left(t-r\right)}}\left\Vert\xi_1\left(r\right)-\xi_2\left(r\right)\right\Vert_{L^2\left(\Omega,X_0\right)}dr\\&\leq{KL_1}\left\Vert\xi_1\left(r\right)-\xi_2\left(r\right)\right\Vert_{\mathcal{C}^+_\tau}e^{\alpha{t}}\int_t^\infty{e^{-\left(\alpha-\gamma\right)r}}dr\\&\leq{KL_1}\left(\alpha-\gamma\right)^{-1}e^{\gamma{t}}\left\Vert\xi_1-\xi_2\right\Vert_{\mathcal{C}^+_\tau}.
		\end{split}
	\end{equation}
	By \eqref{eq2} and \eqref{eq2.9}, we have, for $t\geq\tau$,
	\begin{equation}\label{eq4.20}
		\begin{split}
			\mathcal{I}_2&=\mathbb{E}\left[\left(\int_t^\infty{T_{A_{0u}}}\left(t-r\right)\Pi_u\left(\sigma\left(\xi_1\left(r\right)\right)-\sigma\left(\xi_2\left(r\right)\right)\right)dW\left(r\right)\right)^2\right]^{\frac{1}{2}}\\&\leq{KL_2}\left(\int_t^\infty{e^{2\alpha\left(t-r\right)}}\mathbb{E}\left[\left\Vert\xi_1\left(r\right)-\xi_2\left(r\right)\right\Vert^2_{L^2\left(\Omega,X_0\right)}\right]dr\right)^{\frac{1}{2}}\\&\leq{KL_2}\left\Vert\xi_1\left(r\right)-\xi_2\left(r\right)\right\Vert_{\mathcal{C}^+_\tau}e^{\alpha{t}}\left(\int_t^\infty{e^{-2\left(\alpha-\gamma\right)r}}dr\right)^{\frac{1}{2}}\\&\leq{KL_2}\left(2\alpha-2\gamma\right)^{-\frac{1}{2}}e^{\gamma{t}}\left\Vert\xi_1-\xi_2\right\Vert_{\mathcal{C}^+_\tau}.
		\end{split}
	\end{equation}
	By Lemma \ref{lem0} and \eqref{eq1}, we have, for $t\geq\tau$ and $\zeta\in\left(\beta,\gamma\right)$, 
	\begin{equation}\label{eq4.21}
		\begin{split}
			\mathcal{I}_3&=\lim_{\lambda\to+\infty}\int_\tau^t{\left\Vert{T_{A_{0s}}\left(t-r\right)}\lambda{R_\lambda\left(A_s\right)}\Pi_s\left(F\left(\xi_1\left(r\right)\right)-F\left(\xi_2\left(r\right)\right)\right)\right\Vert}_{L^2\left(\Omega,X_0\right)}dr
			% \\
			% &=\lim_{\lambda\to+\infty}\int_0^{t-\tau}{\Vert{T_{A_{0s}}\left(t-l-\tau\right)\lambda{\left(\lambda{I}-A_s\right)}^{-1}\Pi_s\left(F\left(\xi_1\left(l+\tau\right)\right)-F\left(\xi_2\left(l+\tau\right)\right)\right)}\Vert}_{L^2\left(\Omega,X_0\right)}{dl}
			\\&\leq{L_1}C_\zeta\underset{s\in[0,t-\tau]}{\sup}e^{\zeta\left(t-\tau-s\right)}{\left\Vert\xi_1\left(s+\tau\right)-\xi_2\left(s+\tau\right)\right\Vert}_{L^2\left(\Omega,X_0\right)}\\
			&\leq{L_1}C_\zeta{e^{\gamma{t}}}{\left\Vert\xi_1-\xi_2\right\Vert}_{\mathcal{C}^+_\tau}.
		\end{split}
	\end{equation}
	By Lemma \ref{lem0} and \eqref{eq2}, we have, for $t\geq\tau$ and $\zeta\in\left(\beta,\gamma\right)$, 
	\begin{equation}\label{eq4.22}
		\begin{split}
			\mathcal{I}_4&={\left\Vert\lim_{\lambda\to+\infty}\int_\tau^tT_{A_{0s}}\left(t-r\right)\lambda{R_\lambda\left(A_s\right)}\Pi_s\left(\sigma\left(\xi_1\left(r\right)\right)-\sigma\left(\xi_2\left(r\right)\right)\right))dW\left(r\right)\right\Vert}_{L^2\left(\Omega,X_0\right)}
			%   \\&=\mathbb{E}\left[\left(\lim_{\lambda\to+\infty}\int_0^{t-\tau}T_{A_{0s}}\left(t-l-\tau\right)\lambda{\left(\lambda{I}-A_s\right)}^{-1}\Pi_s\left(\sigma\left(\xi_1\left(l+\tau\right)\right)-\sigma\left(\xi_2\left(l+\tau\right)\right)\right)dW\left(l\right)\right)^2\right]^{\frac{1}{2}}
			\\&\leq{C_\zeta}\underset{s\in[0,t-\tau]}{\sup}e^{\zeta\left(t-\tau-s\right)}\mathbb{E}\left[{\left\Vert\sigma\left(\xi_1\left(s+\tau\right)\right)-\sigma\left(\xi_2\left(s+\tau\right)\right)\right\Vert}^2_{\mathcal{L}_2\left(Y_0,X\right)}\right]^{\frac{1}{2}}\\&\leq{L_2}C_\zeta{e^{\gamma{t}}}{\left\Vert\xi_1-\xi_2\right\Vert}_{\mathcal{C}^+_\tau}.
		\end{split}
	\end{equation}
	By combining \eqref{eq4.18}-\eqref{eq4.22}, we have, for $t\geq\tau$,
	\begin{equation}\label{eq4.23}
		\begin{split}
			&\left\Vert{\xi_1\left(t\right)-\xi_2\left(t\right)}\right\Vert_{\mathcal{C}^+_\tau}\\&\leq{K\left(L_1\left(\alpha-\gamma\right)^{-1}+L_2\left(2\alpha-2\gamma\right)^{-\frac{1}{2}}+L_1C_\zeta+L_2C_\zeta\right)}\left\Vert{\xi_1\left(t\right)-\xi_2\left(t\right)}\right\Vert_{\mathcal{C}^+_\tau}.
		\end{split}
	\end{equation}
	By \eqref{eq4.17} and \eqref{eq4.23}, we get that for $t\geq\tau$, $\xi_1\left(t\right)=\xi_2\left(t\right)$. So such $\xi$ is unique. Given $\tau\in\mathbb{R}$ and $x\in{X_{0s}^\tau}$, denote by $\hat{\xi}\left(x\right)$ the unique solution of \eqref{eq4.2} through the above steps. Let ${K\left(L_1\left(\alpha-\gamma\right)^{-1}+L_2\left(2\alpha-2\gamma\right)^{-\frac{1}{2}}+L_1C_\zeta+L_2C_\zeta\right)}=\delta$. By \eqref{eq4.23}, for $x_1,x_2\in{X_{0s}^\tau}$, we have, for $t\geq\tau$,
	\begin{equation}\label{eq4.24}
		\begin{split}
			&\left\Vert{\hat\xi\left(x_1\right)-\hat\xi\left(x_2\right)}\right\Vert_{\mathcal{C}_\tau^+}\\&\leq\underset{t\geq\tau}{\sup}~e^{-\gamma{t}}\left\Vert{T_{A_{0s}}\left(t-\tau\right)\left(x_1-x_2\right)}\right\Vert_{L^2\left(\Omega,X_0\right)}+\delta{\left\Vert\hat\xi\left(x_1\right)-\hat\xi\left(x_2\right)\right\Vert}_{\mathcal{C}^+_\tau}\\&\leq{K}e^{-\beta\tau}\underset{t\geq\tau}{\sup}~e^{\left(\beta-\gamma\right){t}}\left\Vert{x_1-x_2}\right\Vert_{L^2\left(\Omega,X_0\right)}+\delta{\left\Vert\hat\xi\left(x_1\right)-\hat\xi\left(x_2\right)\right\Vert}_{\mathcal{C}^+_\tau}\\&\leq{Ke^{-\gamma\tau}}\left\Vert{x_1-x_2}\right\Vert_{L^2\left(\Omega,X_0\right)}+\delta{\left\Vert\hat\xi\left(x_1\right)-\hat\xi\left(x_2\right)\right\Vert}_{\mathcal{C}^+_\tau}.
		\end{split}
	\end{equation}
	Thus by \eqref{eq4.24}, for all $x_1,x_2\in{X_{0s}^\tau}$, we have
	\begin{equation}\label{eq4.25}
		{\left\Vert{\hat\xi\left(x_1\right)\left(\tau\right)-\hat\xi\left(x_2\right)\left(\tau\right)}\right\Vert}_{L^2\left(\Omega,X_0\right)}\leq{K}\left(1-\delta\right)^{-1}\left\Vert{x_1-x_2}\right\Vert_{L^2\left(\Omega,X_0\right)}.  
	\end{equation}
	\noindent {\bf Step 2.} We prove that the mean-square stable invariant set is given by a graph of a Lipschitz continuous map.
	Let $h^s\left(x,\tau\right)=\Pi_{0u}\hat\xi\left(x\right)\left(\tau\right)$. Let $s\to{+\infty}$ and $t=\tau$ in \eqref{eq4.11}, we have
	\begin{equation}\label{eq4.26}
		\begin{split}
			h^s\left(x,\tau\right)=&-\int_\tau^\infty{T_{A_{0u}}}\left(\tau-r\right)\Pi_u{F\left(\hat\xi\left(x\right)\left(r\right)\right)}dr\\&-\int_\tau^\infty{T_{A_{0u}}}\left(\tau-r\right)\Pi_u{\sigma\left(\hat\xi\left(x\right)\left(r\right)\right)}dW\left(r\right).
		\end{split}
	\end{equation}
	Indeed, by \eqref{eq4.19}, \eqref{eq4.20} and \eqref{eq4.25}, it follows that for all $x_1,x_2\in{X_{0s}^\tau}$, 
	\begin{equation}
		\label{eq4.27}
		\begin{split}
			&{\left\Vert{h^s\left(x_1,\tau\right)-h^s\left(x_2,\tau\right)}\right\Vert}_{L^2\left(\Omega,X_0\right)}\\&~~\leq{K^2\left(L_1\left(\alpha-\gamma\right)^{-1}+L_2\left(2\alpha-2\gamma\right)^{-\frac{1}{2}}\right)}{\left(1-\delta\right)^{-1}}\left\Vert{x_1-x_2}\right\Vert_{L^2\left(\Omega,X_0\right)}.
		\end{split}
	\end{equation}
	By the definition of $h^s\left(\cdot,\tau\right)$ and \eqref{eq3.29}, we get that $h^s\left(\cdot,\tau\right):X_{0s}^\tau\to{L^2\left(\Omega,\mathcal{F}_\tau;X_{0u}\right)}$ is a Lipschitz continuous map. 
	Then by \eqref{eq4.2} and \eqref{eq4.26}, we have for all $\tau\in\mathbb{R}$ and $x\in{X_{0s}^\tau}$,
	\begin{equation}\label{eq4.28}
		\hat\xi\left(x\right)\left(\tau\right)=x+h^s\left(x,\tau\right).
	\end{equation}
	By the definition of $\mathcal{M}^s\left(\tau\right)$ \eqref{eq4.16} and Lemma \ref{lem4.1}, $\xi\left(\tau\right)\in\mathcal{M}^s\left(\tau\right)$ if and only if there exists $x\in
	{X_{0s}^\tau}$ such that $\xi\left(\tau\right)=x+h^s\left(x,\tau\right)$. Therefore, we have
	\begin{equation*}
		\mathcal{M}^s\left(\tau\right)=\left\{x+h^s\left(x,\tau\right):x\in{X_{0s}^\tau}\right\}.
	\end{equation*}
	\noindent {\bf Step 3.} Given $\tau\in\mathbb{R}$, we prove that $\mathcal{M}^s\left(\tau\right)$ is invariant under $\Phi$. By Definition \ref{inv}, we need to show that for each $\xi\left(\tau\right)\in\mathcal{M}^s\left(\tau\right)$ and $t>0$, $\Phi\left(t,\tau\right)\xi\left(\tau\right)\in\mathcal{M}^s\left(\tau+t\right)$. By \eqref{eq4.15}, for $\xi\left(\tau\right)\in\mathcal{M}^s\left(\tau\right)$, $\xi$ is an integrated solution of \eqref{eq1.1} on $[\tau,+\infty)$. It is easy to prove that $\xi$ also belongs to $\mathcal{C}^+_{\tau+t}$ and is an integrated solution of \eqref{eq1.1} on $[\tau+t,+\infty)$. So by \eqref{eq4.15}, $\xi\left(\tau+t\right)\in\mathcal{M}^s\left(\tau+t\right)$. Since $\xi$ is an integrated solution of \eqref{eq1.1} on $[\tau,+\infty)$, we find that 
	\begin{equation*}
		\Phi\left(t,\tau\right)\left(\xi\left(\tau\right)\right)=u\left(t+\tau,\tau,\xi\left(\tau\right)\right)=\xi\left(t+\tau\right).
	\end{equation*}
	So $\Phi\left(t,\tau\right)\left(\xi\left(\tau\right)\right)\in\mathcal{M}^s\left(\tau+t\right)$, which implies $\Phi\left(t,\tau\right)\mathcal{M}^s\left(\tau\right)\subseteq\mathcal{M}^s\left(\tau+t\right)$. So $\mathcal{M}^s\left(\tau\right)$ is invariant under $\Phi$.
\end{proof}

\section{Example}
Consider the following stochastic Stratonovich parabolic partial equation
\begin{equation}\label{eq8.1}
\left\lbrace
\begin{split}
&\frac{\partial u\left(t,x\right)}{\partial t}=\left(\frac{\partial^2 u\left(t,x\right)}{\partial x^2} +\frac{{\pi}^2}{2}u\left(t,x\right)+g_0\left(u\left(t,x\right) \right)\right)dt+\sigma\left(u\left(t,x\right)\right) dW\left(t\right),\\
&-\frac{\partial u\left(t,0\right)}{\partial x}=g_1\left(u\left(t,\cdot\right)\right),\\
&\frac{\partial u\left(t,1\right)}{\partial x}=g_2\left(u\left(t,\cdot\right)\right),\\
&u\left(0,\cdot\right)=u_0\in L^2\left(0,1\right).
\end{split}
\right.
\end{equation}
where $x\in\left(0,1\right), g_0:L^2\left(0,1\right)\to L^2\left(0,1\right)$, $g_1, g_2:L^2\left(0,1\right)\to\mathbb{R}$ are arbitrary nonlinear functions, which are assumed to be globally Lipschitz continuous and $g_0\left(0\right)=g_1\left(0\right)=g_2\left(0\right)=0$. Herein, $W$ is a two-sided cylindrical Wiener process on a complete filtered probability space $\left({\Omega}, \mathcal{F}, \{{\mathcal{F}_t}\}_{t\in\mathbb{R}}, \mathbb{P}\right)$ and $L^2\left(0,1\right)$ is the usual $L^2$-space on the interval $\left(0,1\right)$. $\sigma$ is a nonlinear function. In order to incorporate the boundary conditions, we define
\begin{equation*}
\mathcal{Y}=\mathbb{R}\times\mathbb{R} \times L^2\left(0,1\right) \text{~~and~~} \mathcal{Y}_0=\left\{0\right\}\times\left\{0\right\} \times L^2\left(0,1\right).
\end{equation*}
which are Banach spaces equipped with the usual product norm. Let us denote
\begin{equation*}
\mathcal{A}
\begin{pmatrix}
0\\
0\\
\varphi
\end{pmatrix}=
\begin{pmatrix}
\varphi'\left(0\right)\\
-\varphi'\left(1\right)\\
\varphi''
\end{pmatrix}
\end{equation*}
with domain $D\left(\mathcal{A}\right)=\left\{0\right\}\times\left\{0\right\} \times W^{2,2}\left(0,1\right)$ and its part
\begin{equation*}
\mathcal{A}_0
\begin{pmatrix}
0\\
0\\
\varphi
\end{pmatrix}=
\begin{pmatrix}
0\\
0\\
\varphi''
\end{pmatrix}
\end{equation*}
with domain $D\left(\mathcal{A}_0\right)=\left\{0\right\}\times\left\{0\right\} \times \left\{\varphi\in W^{2,2}\left(0,1\right): \varphi'\left(0\right)=\varphi'\left(1\right)=0\right\}$, where $W^{2,2}\left(0,1\right)$ is the usual Sobolev space. Notice that
\begin{equation*}
 \mathcal{Y}_0=\overline{D\left(\mathcal{A}\right)}=\left\{0\right\}\times\left\{0\right\} \times L^2\left(0,1\right)\neq\mathcal{Y}
\end{equation*}
and take
\begin{equation*}
F
\begin{pmatrix}
0\\
0\\
\varphi
\end{pmatrix}=
\begin{pmatrix}
g_0\left(\varphi\right)\\
g_1\left(\varphi\right)\\
g_2\left(\varphi\right)
\end{pmatrix},\quad u^*\left(t,\cdot\right)=
\begin{pmatrix}
0\\
0\\
u\left(t,\cdot\right)
\end{pmatrix},\quad \omega\left(t\right)=
\begin{pmatrix}
0\\
0\\
W\left(t\right)
\end{pmatrix} .
\end{equation*}
Then we can rewrite \eqref{eq8.1} as
\begin{equation*}
\left\lbrace
\begin{split}
&du^*\left(t\right)= \left(\left(\mathcal{A}+\frac{1}{2}{\pi}^2I\right)u^*\left(t\right)+F\left(u^*\left(t\right)\right)\right)dt+\sigma\left(u^*\left(t\right)\right)d\omega\left(t\right),\\
&u^*\left(0\right)=u^*_0\in \mathcal{Y}_0.
\end{split}
\right.
\end{equation*}
Then $F:L^2\left(\Omega,\mathcal{Y}_0\right)\to{L^2\left(\Omega,\mathcal{Y}\right)}$ is globally Lipschitz continuous by the Lipschitz continuity of $g_i, i=0, 1, 2$. Assume $\sigma$ is $C^1_b$ and satisfies $\sigma\left(0\right)=0$. Then Assumption \ref{as1} is satisfied.
According to \cite[Lemma 6.1 \& Lemma 6.3]{Magal2016} that the linear operator $\mathcal{A}_0$ is the infintesimal generator of ${\left(T_{\mathcal{A}_0}\left(t\right)\right)}_{t\geq0}$ a $C_0$-semigroup on $\mathcal{Y}_0$ and the non-densely defined operator $\mathcal{A}$ is 3/4-almost sectorial, namely, its resolvent satisfies
\begin{equation*}
0 < \liminf_{\lambda\to\infty}\lambda^{\frac34}\left\|\left(\lambda I-\mathcal{A}\right)^{-1}\right\|_{\mathcal{L}\left(x\right)}\le
 \limsup_{\lambda\to\infty}\lambda^{\frac34}\left\|\left(\lambda I-\mathcal{A}\right)^{-1}\right\|_{\mathcal{L}\left(x\right)}<\infty.
\end{equation*}
Then \cite[Lemma 8.5]{PM2009} guarantees Assumption \ref{as2.2}. Moreover, the spectrum of $\mathcal{A}_0$ is given by
\begin{equation*}
	\sigma\left(\mathcal{A}_0\right)=\left\{-\left(\pi{k}\right)^2:k\in\mathbb{N}\right\}.
\end{equation*}
Moreover, $\left(\mathcal{A}+\frac{1}{2}\pi^2{I}\right)_0$, the part of $\left(\mathcal{A}+\frac{1}{2}\pi^2{I}\right)$, is the infinitesimal generator of a $C_0$-semigroup on $\mathcal{Y}_0$ denoted by ${\left(T_{\left(\mathcal{A}+\frac{1}{2}\pi^2{I}\right)_0}\left(t\right)\right)}_{t\geq0}$. Also, according to \cite[Lemma 6.4] {Magal2016} and \cite[Proposition 2.5]{PM2007,PM2009}, we get that Assumption \ref{as2.1} is satisfied, which means $\mathcal{A}$ does not satisfy the Hille-Yosida condition but $\left(\mathcal{A}+\frac{1}{2}\pi^2{I}\right)$ generates a integrated semigroup ${\left(S_{\left(\mathcal{A}+\frac{1}{2}\pi^2{I}\right)}\left(t\right)\right)}_{t\geq0}$. So we next check Assumption \ref{as2.5}. In fact, by \cite[Lemma 6.2]{Magal2016} and \cite[Lemma 2.1]{PM_2009}, we have $\sigma\left(\mathcal{A}_0\right)=\sigma\left(\mathcal{A}\right)$, then
\begin{equation*}
\begin{split}
	\sigma\left({\mathcal{A}}+\frac{1}{2}\pi^2I\right)&=\sigma\left(\mathcal{A}_0+\frac{1}{2}\pi^2I\right)\\&=\left\{-\left(\pi{k}\right)^2+\frac{1}{2}\pi^2:k\in\mathbb{N}\right\}\\&=\left\{\frac{1}{2}\pi^2, -\frac{1}{2}\pi^2, -\frac{7}{2}\pi^2, -\frac{17}{2}\pi^2,...\right\},
\end{split}
\end{equation*}
and each enginvalue $\lambda_k=\left(\frac{1}{2}-k^2\right)\pi^2$ corresponding to the enginfunction
\begin{equation*}
	\psi_k\left(x\right)=\sin\left(\pi{kx}\right).
\end{equation*}
By Remark \ref{rem}, we could take $\alpha=\beta=\frac{1}{2}\pi^2-\varepsilon^*$ for any $\varepsilon^*\in\left(0,\frac{1}{2}\pi^2\right)$. Thus the spectrum of $\mathcal{A}_0+\frac{1}{2}\pi^2I$ could be split into two parts $\sigma^{0s}=\{\left(\frac{1}{2}-k^2\right)\pi^2:k=1,2,...\},\sigma^{0u}=\{\frac{1}{2}\pi^2\}$. So $\mathcal{A}_0+\frac{1}{2}\pi^2I$ satisfies the exponential dichotomy condition and the Assumption \ref{as2.5} is satisfied. The stable subspace $\mathcal{Y}_{0s}$ and unstable subspace $\mathcal{Y}_{0u}$ of $\mathcal{Y}_0$ are $\text{span}\left\{\sin\left(\pi{nx}\right),n=1,2,...\right\}$, and $\left\{0\right\}$, respectively. 

Therefore, the results obtained in Theorems \ref{thm3.2} and \ref{thm4.2} are applicable and thus justify the existence of mean-square random unstable manifolds and stable sets
for \eqref{eq8.1}.

\section{Concluding remarks}\label{sec5}
Under the suitable assumptions, we have established the existence of mean-square random unstable invariant manifolds and mean-square random stable invariant set of \eqref{eq1.1} that are given by graphs of Lipschitz maps. 

We conclude this paper by discussing several possible extensions of our result. The first extension is straightforward and can be obtained by simply rereading the paper carefully. The two other extensions are less obvious and merit further investigation.

\begin{itemize}
	\item [$\left(\mathrm{1}\right)$] One can consider the existence of center manifold of \eqref{eq1.1} if one additionally has a center subspace, namely if there exist eigenvalues of $A_0$ with real part greater than zero. Under the modified assumption of exponential dichotomy for $A_0$, one can obtain the existence of mean-square center-unstable invariant manifold and mean-square center-stable invariant set of $\eqref{eq1.1}$.
	
	\item [$\left(\mathrm{2}\right)$] We need to point out that the usual assumptions $F:L^2\left(\Omega,X_0\right)\to{L^2\left(\Omega,X\right)}$ and $\sigma:L^2\left(\Omega,X_0\right)\to{L^2\left(\Omega,\mathcal{L}_2\left(Y_0,X\right)\right)}$ are $C^k$ (for $k\geq1$) can not guarantee the smoothness of random invariant manifolds or invariant sets. The main difficulty that one has to face arises from the fact that the required gap condition contains some $C_\kappa$ like constant since Young's convolution inequality is not applicable. There is no conceptual obstruction to the use of the method of proof presented in this paper in that situation, but new estimates are required.
	
	\item [$\left(\mathrm{3}\right)$] One can weaken the hyperbolicity of invariant manifolds in the present paper. We refer to \cite{PJ2013,WL2021} for implementations of this idea in the framework of random dynamical systems. Herein,  Bates \emph{et al.} \cite{PJ2013} studied normally hyperbolic invariant manifolds for random dynamical systems and the persistence of normal hyperbolicity under random perturbations. Zhou and Zhang \cite{WL2021} weakened the results to the nonuniformly normally hyperbolic case. A natural question is whether the results of the present paper also apply to these weakened situations.
\end{itemize}

\section*{Acknowledgments}
Zeng is partially supported by the National Natural Science Foundation of China (No. 11871225), Guangdong Basic and Applied Basic Research Foundation (No. 2019A1515011350) and Guangzhou Basic and Applied Basic Research Foundation (No.
202 11911530750).
 Huang is partially supported by the National Natural Science Foundation of China (No. 11771449).

\bibliographystyle{elsarticle-num}
\bibliography{ref}

\begin{thebibliography}{10}
\expandafter\ifx\csname url\endcsname\relax
  \def\url#1{\texttt{#1}}\fi
\expandafter\ifx\csname urlprefix\endcsname\relax\def\urlprefix{URL }\fi
\expandafter\ifx\csname href\endcsname\relax
  \def\href#1#2{#2} \def\path#1{#1}\fi

\bibitem{WA1987}
W.~Arendt, Resolvent positive operators, Proc. London Math. Soc. s3-54~(2)
  (1987) 321--349.

\bibitem{WA2001}
W.~Arendt, C.~Batty, M.~Hieber, F.~Neubrander, Vector-Valued Laplace Transforms
  and Cauchy Problems, Birkh{\"a}user, Basel, 2001.

\bibitem{DS1987}
G.~{Da Prato}, E.~Sinestrari, Differential operators with non-dense domain,
  Ann. Scuola Norm-Sci. 14~(2) (1987) 285--344.

\bibitem{HS1990_}
H.~Thieme, Integrated semigroups and integrated solutions to abstract {Cauchy}
  problems, J. Math. Anal. Appl. 152~(2) (1990) 416--447.

\bibitem{HS1990}
H.~Thieme, Semiflows generated by {Lipschitz} perturbations of non-densely
  defined operators, Differential Integral Equations 3~(6) (1990) 1035--1066.

\bibitem{PM2007}
P.~Magal, S.~Ruan, On integrated semigroups and age-structured models in
  $\mathcal{L}^p$ space, Differential Integral Equations 20~(2) (2007)
  197--239.

\bibitem{PM_2009}
P.~Magal, S.~Ruan, Center manifolds for semilinear equations with non-dense
  domain and applications to {Hopf} bifurcation in age structured models, Mem.
  Amer. Math. Soc. 202~(951) (2009) 71pp.

\bibitem{PM2009}
P.~Magal, S.~Ruan, On semilinear {Cauchy} problems with non-dense domain, Adv.
  Difference Equations 14~(11-12) (2009) 1041--1084.

\bibitem{AN2020}
A.~Neam\c{t}u, Random invariant manifolds for ill-posed stochastic evolution
  equations, Stoch. Dyn. 20~(2) (2020) 2050013.

\bibitem{ZS2021}
C.~Zeng, J.~Shen, Invariant foliations for stochastic partial differential
  equations with non-dense domain, Proc. Amer. Math. Soc.Accepted (2021).

\bibitem{LZ2021}
Z.~Li, C.~Zeng, Center manifolds for ill-posed stochastic evolution equations,
  Discrete Contin. Dyn. Syst. Ser. BAccepted (2021).

\bibitem{Mh1999}
S.-E.~A. Mohammed, M.~K.~R. Scheutzow, The stable manifold theorem for
  stochastic differential equations, Ann. Probab. 27~(2) (1999) 615--652.

\bibitem{Duan2004}
J.~Duan, K.~Lu, B.~Schmalfuss, Smooth stable and unstable manifolds for
  stochastic evolutionary equations, J. Dynam. Differential Equations 16~(4)
  (2004) 949--972.

\bibitem{KB2007}
K.~Lu, B.~Schmalfuss, Invariant manifolds for stochastic wave equation, J.
  Differential Equations 236~(2) (2007) 460--492.

\bibitem{Cara2010}
T.~Caraballo, J.~Duan, K.~Lu, B.~Schmalfuss, Invariant manifolds for random and
  stochastic partial differential equations, Adv. Nonlinear Stud. 10~(1) (2010)
  23--52.

\bibitem{LL2010}
Z.~Lian, K.~Lu, Lyapunov exponents and invariant manifolds for random dynamical
  systems in a banach space introduction, Mem. Amer. Math. Soc. 206~(967)
  (2010) 106pp.

\bibitem{PJ2013}
J.~Li, K.~Lu, P.~Bates, Normally hyperbolic invariant manifolds for random
  dynamical systems: Part {I} - persistence, Trans. Amer. Math. Soc. 365 (2013)
  5933--5966.

\bibitem{Chen2015}
X.~Chen, A.~J. Roberts, J.~Duan, Center manifolds for stochastic evolution
  equations, J. Difference Equ. Appl. 21~(7) (2015) 606--632.

\bibitem{Shi2020}
L.~Shi, Smooth convergence of random center manifolds for {SPDEs} in varying
  phase spaces, J. Differential Equations 269~(3) (2020) 1963--2011.

\bibitem{Cara2013}
T.~Caraballo, J.~A. Langa, J.~Robinson, A stochastic pitchfork bifurcation in a
  reaction-diffusion equation, R. Soc. Lond. Proc. Ser. A 457~(2013) (2013)
  2041--2061.

\bibitem{Duan2003}
J.~Duan, K.~Lu, B.~Schmalfuss, Invariant manifolds for stochastic partial
  differential equations, Ann. Probab. 31~(4) (2003) 2109--2135.

\bibitem{GALS2010}
M.~Garrido-Atienza, K.~Lu, B.~Schmalfuss, Unstable invariant manifolds for
  stochastic {PDEs} driven by a fractional {Brownian} motion, J. Differential
  Equations 248~(7) (2010) 1637--1667.

\bibitem{KN2021}
A.~Neam\c{t}u, C.~Kuehn, Rough center manifolds, SIAM J. Math. Anal. 53~(4)
  (2021) 3912--3957.

\bibitem{LK2012}
P.~Kloeden, T.~Lorenz, Mean-square random dynamical systems, J. Differential
  Equations 253~(5) (2012) 1422--1438.

\bibitem{BA2021}
B.~Wang, Mean-square random invariant manifolds for stochastic differential
  equations, Discrete Contin. Dyn. Syst. Ser. A 41~(3) (2021) 1449--1468.

\bibitem{Magal2016}
P.~Magal, O.~Seydi, Variation of constants formula and exponential dichotomy
  for non-autonomous non densely defined {Cauchy} problems, Can. J.
  Math.Accepted (2020).

\bibitem{APazy1983}
A.~Pazy, Semigroups of Linear Operator and Applications to Partial Differential
  Equations, Springer, 1983.

\bibitem{TG1993}
T.~Gallay, A center-stable manifold theorem for differential equations in
  {Banach} spaces, Comm. Math. Phys.. 152~(2) (1993) 249--268.

\bibitem{YS1991}
Y.~Hu, S.~Peng, Adapted solution of a backward semilinear stochastic evolution
  equation, Stocha. Anal. Appl. 9 (1991) 445--459.

\bibitem{WL2021}
L.~Zhou, W.~Zhang, Approximative dichotomy and persistence of nonuniformly
  normally hyperbolic invariant manifolds in {Banach} spaces, J. Differential
  Equations 274 (2021) 35--126.

\end{thebibliography}



%%       \citet{<label>} ==> Jones et al. [21]
%%       \citep{<label>} ==> [21]
%%



\end{document}